\documentclass[a4paper]{amsart}
\usepackage{graphicx}
\usepackage{amssymb}
\usepackage{amsmath}
\usepackage{amsthm}
\usepackage{amscd}
\usepackage[all,2cell]{xy}

\usepackage[pagebackref,colorlinks]{hyperref}

\UseAllTwocells \SilentMatrices
\newtheorem{thm}{Theorem}[section]

\newtheorem{cor}[thm]{Corollary}

\newtheorem{lem}[thm]{Lemma}
\newtheorem{exm}[thm]{Example}

\newtheorem{prop}[thm]{Proposition}

\theoremstyle{definition}

\theoremstyle{remark}
\newtheorem{rem}[thm]{\bf Remark}
\numberwithin{equation}{section}

\begin{document}
\title[Noncommutative Kn\"{o}rrer  periodicity via equivariantization]{Noncommutative Kn\"{o}rrer periodicity via equivariantization}
\author[Chen, Wu] {Xiao-Wu Chen$^*$, Wenchao Wu}

\makeatletter
\@namedef{subjclassname@2020}{\textup{2020} Mathematics Subject Classification}
\makeatother

\thanks{$^*$ The corresponding author}
\date{\today}

\thanks{xwchen$\symbol{64}$mail.ustc.edu.cn, wuwch20$\symbol{64}$mail.ustc.edu.cn}

\begin{abstract}
We establish  noncommutative Kn\"{o}rrer periodicity for projective-module factorizations over an arbitrary ring, using the equivariantization theory with respect to various actions by a cyclic group of order two. We obtain an explicit quasi-inverse of the periodicity. We  compare the periodicity with a certain tensor functor between big singularity categories.
\end{abstract}

\maketitle

\dedicatory{}%
\commby{}%

\section{Introduction}

Kn\"{o}rrer's periodicity theorem \cite{Kno} plays a central role in the classification of simple hypersurface singularities \cite{BGS, Sol}. It relates the stable categories of \emph{matrix factorizations} \cite{Eis} over rings with different  Krull dimension.  Kn\"{o}rrer periodicity is related to Bott periodicity  \cite{Br} in topological K-theory, and represents  a certain  quantum symmetry in Landau-Ginzburg models \cite{HHP}.

There are geometric versions of Kn\"{o}rrer's periodicity theorem in \cite{Orl, Shi}. We are interested in
noncommutative  Kn\"{o}rrer periodicity. The concept of matrix factorizations over noncommutative rings is due to \cite{CCKM, MU21}. Then Kn\"{o}rrer's periodicity theorem for noncommutative rings is established in \cite{CKMW, MU}.

We mention that $C_2$-actions on rings and $\mathbb{Z}_2$-graded rings appear naturally in the study of matrix factorizations \cite{Kno, Sol}. Here and later, $C_2$ will denote a cyclic group of order two, and $\mathbb{Z}_2$-graded rings will be called super rings. Therefore, it is natural to investigate the categories of  matrix factorizations using the equivariantization theory \cite{Del} with respect to various $C_2$-actions; compare \cite{SY}.

The goal of this paper is two-fold. The first is to  give a more conceptual proof of noncommutative  Kn\"{o}rrer periodicity, using $C_2$-equivariantization. The second is to remove the usual assumptions on the ring \cite{Kno, CKMW, MU} such as the regularity, noetherianness, completeness or gradedness. To do so, we consider \emph{projective-module factorizations} \cite{Chen24}, which are infinite analogues \cite{Dyck, Bird} of matrix factorizations.

Let $A$ be an arbitrary ring with unit. Fix an element $\omega\in A$ and an automorphism $\sigma$ on $A$ such that
$$\sigma(\omega)=\omega, \mbox{ and } \omega a=\sigma(a) \omega  \mbox{ for all } a\in A.$$
For convenience, such a triple $(A, \omega, \sigma)$ will be called an \emph{nc-triple}. The element $\omega$ plays the role of potentials in  Landau-Ginzburg models \cite{Orl, HHP}.

Using the nc-triple $(A, \omega, \sigma)$ above, the category ${\mathbf{PF}}(A; \omega)$ of projective-module factorizations of $\omega$ over $A$ is defined \cite{Chen24}. Its stable category $\underline{\mathbf{PF}}(A; \omega)$ is canonically triangulated \cite{Hap}. Similarly, we have the category ${\mathbf{MF}}(A; \omega)$ of matrix factorizations and the stable category $\underline{\mathbf{MF}}(A; \omega)$; see  \cite{CCKM, MU21}.

We fix another automorphism $\tau$ on $A$ such that
$$\tau(\omega)=\omega \mbox{ and } \tau^2=\sigma.$$
Consider the skew polynomial rings $A_1=A[x; \tau]$ and $A_2=A_1[y; \tau_1]$. Here, we denote by $\tau_1$  (\emph{resp}. $\sigma_1$) the  automorphism on $A_1$, which  extends $\tau$  (\emph{resp}. $\sigma$) and fixes $x$.  Denote by $\sigma_2$ the automorphism on $A_2$, which extends $\sigma$ and fixes both $x$ and $y$. Then $(A_2, y^2-x^2+\omega, \sigma_2)$ is an nc-triple. Using it, we form the stable category $\underline{\mathbf{PF}}(A_2; y^2-x^2 + \omega)$ of projective-module factorizations of $y^2-x^2 +\omega$ over $A_2$.  We say that the integer $2$ is \emph{invertible} in $A$ if the element $1_A+1_A$ is invertible in $A$.

The main result is  noncommutative Kn\"{o}rrer periodicity  for projective-module factorizations; see Theorem~\ref{thm:main} and Proposition~\ref{prop:inverse}.

\vskip 5pt

\noindent {\bf Theorem I}. \; \emph{Consider the nc-triples $(A, \omega, \sigma)$ and $(A_2, y^2-x^2+\omega, \sigma_2)$ above. Assume that $2$ is invertible in $A$. Then there is an explicit triangle equivalence
$${\rm Kn}\colon \underline{\mathbf{PF}}(A; \omega) \stackrel{\sim}\longrightarrow \underline{\mathbf{PF}}(A_2; y^2-x^2+\omega)$$
with an explicit quasi-inverse $A_1\otimes_{A_2}-$. The equivalence ${\rm Kn}$  restricts to a triangle equivalence up to retracts
$$\underline{\mathbf{MF}}(A; \omega) \longrightarrow \underline{\mathbf{MF}}(A_2; y^2-x^2+\omega)$$
between the stable categories of matrix factorizations.
}

\vskip 5pt

The quasi-inverse  $A_1\otimes_{A_2}-$ above is brand new,  since it does not restrict to matrix factorizations.  We do not know the denseness of the restricted functor of ${\rm Kn}$ between matrix factorizations, in general. If the stable category $\underline{\mathbf{MF}}(A; \omega)$ is idempotent-split, the denseness holds true, and thus we obtain a triangle equivalence between matrix factorizations.

The key ingredient in  the proof of Theorem~I is the intermediate  nc-triple $(A_1, x^2-\omega, \sigma_1)$. The stable category
$\underline{\mathbf{PF}}(A_1; x^2-\omega)$ of projective-module factorizations of $x^2-\omega$ over $A_1$ admits a $C_2$-action induced by the \emph{swap-twisting endofunctor}; see Section~\ref{sec:PF}.

Consider the quotient ring $\overline{A_1}=A_1/{(x^2-\omega)}$, which is a \emph{noncommutative double branched cover} of $A$; see \cite{Kno, MU21}. Denote by $\overline{A_1}\mbox{-GProj}^{\rm fpd}$ the category of Gorenstein-projective $\overline{A_1}$-modules \cite{ABr, EJ} whose underlying $A$-modules have finite projective dimension. The stable category
$\overline{A_1}\mbox{-}\underline{\rm GProj}^{\rm fpd}$ is canonically triangulated \cite{Hap}. The ring $\overline{A_1}$ admits a \emph{parity automorphism} $g$, which sends $x$ to $-x$. The corresponding \emph{twisting autoequivalence} on $\overline{A_1}\mbox{-}\underline{\rm GProj}^{\rm fpd}$ induces a $C_2$-action.

The following result goes back to \cite{Kno, Yos}; see Theorem~\ref{thm:Kno} and Proposition~\ref{prop:G-PF}. It is used twice in the construction of ${\rm Kn}$ in Theorem~I.
\vskip 5pt

\noindent {\bf Theorem II}. \; \emph{We have an explicit triangle equivalenece
$${\rm SE}\colon  \overline{A_1}\mbox{-}\underline{\rm GProj}^{\rm fpd} \stackrel{\sim}\longrightarrow \underline{\mathbf{PF}}(A_1; x^2-\omega),$$
which is equivariant with respect to the two $C_2$-actions mentioned above.
}

\vskip 5pt

 In Theorem~II, it is crucial that the explicit equivalence ${\rm SE}$, the so-called \emph{scalar-extension functor}, is compatible with these two $C_2$-actions. In particular, ${\rm SE}$ induces an equivalence between the categories of equivariant objects.

The more well-known result in \cite{Eis} implies  that the (zeroth) \emph{cokernel functor}
$${\rm Cok^0}\colon \underline{\mathbf{PF}}(A_1; x^2-\omega) \stackrel{\sim} \longrightarrow \overline{A_1}\mbox{-}\underline{\rm GProj}^{\rm fpd}$$
is a triangle equivalence; compare \cite{Chen24}. The equivalence ${\rm SE}$ above is a quasi-inverse of ${\rm Cok}^0$ up to a twisting autoequivalence; see \cite{Yos}. However, in comparison with ${\rm Cok}^0$,  the functor ${\rm SE}$ seems to be more convenient for us.

Consider the quotient rings $\bar{A}=A/{(\omega)}$ and $\overline{A_2}=A_2/{(y^2-x^2+\omega)}$. The automorphism $\tau$ on $A$ induces an automorphism $\bar{\tau}$ on $\bar{A}$. We form the skew polynomial ring $B=\bar{A}[x; \bar{\tau}]$. It is naturally an $\overline{A_2}$-$\bar{A}$-bimodule; see Section~\ref{sec:compare}.

For any ring $R$, we denote by $\mathbf{D}'_{\rm sg}(R)$ its \emph{big singularity category} \cite{Buc, Orl}. The canonical functor $Q_R\colon R\mbox{-}\underline{\rm Mod}\rightarrow \mathbf{D}'_{\rm sg}(R)$ sends a module to the corresponding stalk complex concentrated in degree zero. We mention the connection of singularity categories to homological mirror symmetry \cite{Orl, Ebel} for Landau-Ginzburg models.

In Theorem~\ref{thm:comparison}, we compare the noncommuative  Kn\"{o}rrer periodicity with the tensor functor $B\otimes_{\bar{A}}-$ between the big singularity categories in  \cite{Orl}.

\vskip 5pt

\noindent {\bf Theorem~III.} \emph{Assume that the element $\omega$ is regular in $A$. Then we have a commutative diagram up to a natural isomorphism.
    \[\xymatrix{
\underline{\mathbf{PF}}(A; \omega) \ar[d]_-{Q_{\bar{A}}\circ {\rm Cok}^0} \ar[rr]^-{\rm Kn} && \underline{\mathbf{PF}}(A_2; y^2-x^2+\omega) \ar[d]^-{Q_{\overline{A_2}}\circ {\rm Cok}^0}\\
    \mathbf{D}'_{\rm sg}(\bar{A})  \ar[rr]^-{B\otimes_{\bar{A}}-} &&  \mathbf{D}'_{\rm sg}(\overline{A_2})
    }\]
Here, ${\rm Kn}$ is the explicit functor in Theorem~I, and  the two ${\rm Cok}^0$'s denote the corresponding cokernel functors.}

\vskip 5pt

The paper is structured  as follows.  In  Section~2, we recall basic facts on $C_2$-equivariantization. We study super modules via $C_2$-equivariantization in Section~3. We relate projective-module factorizations to super Gorenstein-projective modules in Section~\ref{sec:PF}. We prove that the scalar-extension functor {\rm SE} \cite{Kno, Yos}  is $C_2$-equivariant in Section~\ref{sec:SE}.

By combining the results in the previous sections, we obtain the explicit noncommutative  Kn\"{o}rrer periodicity ${\rm Kn}$ for projective-module factorizations  and its quasi-inverse  in Section~\ref{sec:Kn}. We apply the consideration to root categories \cite{Hap87, PX} in Example~\ref{exm:root}. In the final section, we compare the periodicity with a certain tensor functor \cite{Orl} between the relevant big singularity categories.

Throughout this paper, $C_2$ is a fixed cyclic group of order two. By default, modules will always mean left unital modules. For triangulated categories, we refer to \cite{Hap}, and for equivariantization, we refer to \cite{Del, DGNO, CCR}.

\section{Categories with $C_2$-actions and equivariantization}

In this section, we recall basic facts on categories with $C_2$-actions and equivariant objects.

\subsection{} Let $\mathcal{C}$ be a category. Following \cite[Section~3]{CCR},  a \emph{$C_2$-action} on $\mathcal{C}$ is given by a pair $(T, u)$, where $T\colon \mathcal{C}\rightarrow \mathcal{C}$ is an autoequivalence, and $u\colon {\rm Id}_\mathcal{C} \rightarrow T^2$ is a natural isomorphism satisfying $Tu=uT$; compare \cite{Del}. Such a $C_2$-action $(T, u)$ is called \emph{strict} if $T$ is an automorphism satisfying $T^2={\rm Id}_\mathcal{C}$ and $u$ is the identity transformation.

Assume that  $\mathcal{C}$ is endowed with a fixed $C_2$-action $(T, u)$. A \emph{$C_2$-equivariant object} $(X, \alpha)$ consists of an object $X$ in $\mathcal{C}$ and an isomorphism $\alpha\colon X\rightarrow T(X)$ satisfying $T(\alpha)\circ \alpha=u_X$. A morphism $f\colon (X, \alpha)\rightarrow (Y, \beta)$ between two $C_2$-equivariant objects is given by a morphism $f\colon X\rightarrow Y$ in $\mathcal{C}$ satisfying $\beta\circ f=T(f)\circ \alpha$. This gives rise to the category $\mathcal{C}^{(T, u)}$ of $C_2$-equivariant objects. The \emph{forgetful functor}
$$U\colon \mathcal{C}^{(T, u)}\longrightarrow \mathcal{C} $$
sends $(X, \alpha)$ to $X$.

Let $\mathcal{D}$ be another category with a fixed $C_2$-action $(S, v)$. Denote by $\mathcal{D}^{(S, v)}$ the corresponding category of $C_2$-equivariant objects. By a \emph{$C_2$-equivariant functor} $(F,\eta)\colon \mathcal{C}\rightarrow \mathcal{D}$ with respect to these two $C_2$-actions, we mean a functor $F\colon \mathcal{C} \rightarrow \mathcal{D}$ and a natural isomorphism $\eta\colon FT \rightarrow SF$ which satisfy
$$S\eta\circ \eta T\circ Fu=vF.$$
Such a $C_2$-equivariant functor $(F, \eta)$ is called a \emph{$C_2$-equivariant equivalence} if the underlying functor $F$ is an equivalence.

Any $C_2$-equivariant functor $(F, \eta)$ induces a functor between the categories of $C_2$-equivariant objects
\begin{align}\label{fun:equiv}
(F, \eta)^{C_2}\colon \mathcal{C}^{(T, u)}\longrightarrow \mathcal{D}^{(S, v)},
\end{align}
which sends $(X, \alpha)$ to $(F(X), \eta_X\circ F(\alpha))$.

The following fact is well known; see \cite[Lemma~2.1]{CCR}.

\begin{lem}\label{lem:equiva-equiva}
Assume that  $(F, \eta)\colon \mathcal{C}\rightarrow \mathcal{D}$ is a $C_2$-equivariant equivalence. Then the induced functor $(F, \eta)^{C_2}$ is an equivalence. \hfill $\square$
\end{lem}

\subsection{} Assume that $\mathcal{C}$ is an additive category with a fixed $C_2$-action $(T, u)$. Then the category $\mathcal{C}^{(F, u)}$ is also additive.

Each object $X$ in $\mathcal{C}$ gives rise to a $C_2$-equivariant object
$${\rm Ind}(X)=(X\oplus T(X), \begin{pmatrix}
    0 & {\rm Id}_{T(X)}\\
    u_X & 0
\end{pmatrix}).$$
Here, we identify $T(X\oplus T(X))$ with $T(X)\oplus T^2(X)$. Any morphism $f\colon X\rightarrow X'$ in $\mathcal{C}$ yields a morphism
$\begin{pmatrix}
    f & 0\\
     0 & T(f)
\end{pmatrix}\colon {\rm Ind}(X)\rightarrow {\rm Ind}(Y)$ in $\mathcal{C}^{(T, u)}$. Therefore, we have the \emph{induction functor}
$${\rm Ind}\colon \mathcal{C} \longrightarrow \mathcal{C}^{(T, u)}.$$

The following lemma is also well known; see \cite[Subsection~4.2]{CCR}.

\begin{lem}\label{lem:Fro}
We have two adjoint pairs $({\rm Ind}, U)$ and $(U, {\rm Ind})$.
\end{lem}

\begin{proof}
Let $X$ be an object in $\mathcal{C}$ and $(Y, \beta)$ a $C_2$-equivariant object. We have a natural isomorphism
\begin{align}\label{adj:Ind-U}
{\rm Hom}_\mathcal{C}(X, Y)\longrightarrow {\rm Hom}_{\mathcal{C}^{(T, u)}}({\rm Ind}(X), (Y, \beta)),
\end{align}
sending $f$ to $(f, \beta^{-1}\circ T(f))$. This yields the adjoint pair $({\rm Ind}, U)$. For the other one, we refer to \cite[(4.4)]{CCR}.
\end{proof}

Assume that $\mathcal{C}$ is an abelian category. Then so is the category $\mathcal{C}^{(T, u)}$. Moreover, the forgetful functor $U$ is exact.

\begin{lem}\label{lem:en-proj}
    Assume that $\mathcal{C}$ is an abelian category with enough projective objects. Then so is the category $\mathcal{C}^{(T, u)}$.
\end{lem}

\begin{proof}
    By the adjunction isomorphism (\ref{adj:Ind-U}), ${\rm Ind}(P)$ is projective for any projective object $P$ in $\mathcal{C}$. Take any  $C_2$-equivariant object $(Y, \beta)$. By assumption, we take an epimorphism $\pi\colon P\rightarrow Y$ in $\mathcal{C}$ with $P$ projective.  By the isomorphism (\ref{adj:Ind-U}) again, we have the corresponding morphism
    $$(\pi, \beta^{-1}\circ T(\pi))\colon {\rm Ind}(P)\longrightarrow (Y, \beta),$$ which is certainly an epimorphism. This proves that $\mathcal{C}^{(T, u)}$ has enough projective objects.
\end{proof}

\subsection{} Let $\mathcal{A}$ be an additive category. An idempotent $e\colon X\rightarrow X$ is said to be \emph{split}, if there is a factorization $X\stackrel{a}\rightarrow X' \stackrel{b}\rightarrow  X$ of $e$ satisfying ${\rm Id}_{X'}=a\circ b$. In such a situation, the object $X'$ is called a \emph{retract} of $X$. The additive category $\mathcal{A}$ is said to be \emph{idempotent-split} if each idempotent in $\mathcal{A}$ splits.

Let $H\colon \mathcal{A} \rightarrow \mathcal{B}$ be an additive functor between two additive categories. We say that $H$ is an \emph{equivalence up to retracts}, if it is fully faithful and each object $B$ in $\mathcal{B}$ is a retract of $H(A)$ for some object $A$ in $\mathcal{A}$.

The following result is well known; see \cite[Lemma~3.4(3)]{Chen15}.

\begin{lem}\label{lem:e-uptore}
Let $H\colon \mathcal{A} \rightarrow \mathcal{B}$  be  an equivalence up to retracts. Assume that $\mathcal{A}$ is idempotent-split. Then $H$ is an equivalence. \hfill $\square$
\end{lem}

We say that  the integer $2$ is \emph{invertible} in $\mathcal{A}$, if for any morphism $f\colon X\rightarrow Y$, there is a unique morphism $f'$ satisfying $f=2f'$; see \cite[p.255]{RR}. For example, if $\mathcal{A}$ is $\mathbb{K}$-linear over a field $\mathbb{K}$, $2$ is invertible in $\mathcal{A}$ if and only if  the characteristic of  the base field $\mathbb{K}$ is different from $2$.

Assume that $\mathcal{A}$ is an abelian category with enough projective objects. Denote by $\underline{\mathcal{A}}$ its \emph{stable category} \cite{ABr} modulo projective objects. For two objects $X$ and $Y$, its Hom-group in  $\underline{\mathcal{A}}$ is denoted by $\underline{\rm Hom}_\mathcal{A}(X, Y)$. For any morphism $f\colon X\rightarrow Y$ in $\mathcal{A}$, we denote by $\underline{f}$ its image in $\underline{\rm Hom}_\mathcal{A}(X, Y)$.

Let $R$ be a ring. Denote by $R\mbox{-Mod}$ the category of left $R$-modules. The stable category of $R\mbox{-Mod}$ is denoted by $R\mbox{-\underline{\rm Mod}}$. The following well-known fact follows from a general result \cite{Fre}.

\begin{lem}\label{lem:idem-R}
    Let $R$ be a ring. Then the stable category $R\mbox{-\underline{\rm Mod}}$ is always idempotent-split. \hfill $\square$
\end{lem}

In what follows, we assume that $\mathcal{A}$ is an abelian category with enough projective objects. We fix a $C_2$-action $(T, u)$ on $\mathcal{A}$. By Lemma~\ref{lem:en-proj}, the category $\mathcal{A}^{(T, u)}$ is also abelian with enough projective objects. So, we form the stable category $\underline{\mathcal{A}^{(T, u)}}$.

On the other hand, the $C_2$-action $(T, u)$ on $\mathcal{A}$ induces a $C_2$-action on $\underline{\mathcal{A}}$. By abuse of notation, the induced $C_2$-action is still denoted by $(T, u)$. We form the category $(\underline{\mathcal{A}})^{(T, u)}$ of $C_2$-equivariant objects. Moreover, we have the \emph{comparison functor}
$$K\colon \underline{\mathcal{A}^{(T, u)}}  \longrightarrow  (\underline{\mathcal{A}})^{(T, u)}, \; (X, \alpha) \longmapsto (X, \underline{\alpha}).$$

The following result is analogous to \cite[Proposition~4.5]{Chen15} with a very similar proof.

\begin{prop}\label{prop:comp}
Assume further that $2$ is invertible in $\mathcal{A}$. Then the comparison functor
$K$ is an equivalence up to retracts. If moreover $\underline{\mathcal{A}^{(T, u)}}$ is idempotent-split, $K$ is an equivalence.
\end{prop}

\begin{proof}
The functors $U\colon \mathcal{A}^{(T, u)}\rightarrow \mathcal{A}$ and ${\rm Ind}\colon \mathcal{A}\rightarrow \mathcal{A}^{(T, u)}$ induce the corresponding ones $\underline{U}$ and $\underline{\rm Ind}$ between the stable categories. Moreover, they form an adjoint pair $(\underline{\rm Ind}, \underline{U})$. Since $2$ is invertible in $\mathcal{A}$, the functor $U$ is separable; see \cite[Lemma~4.4(1)]{Chen15}. It follows that $\underline{U}$ is also separable.

We observe that the monad on $\underline{\mathcal{A}}$ defined by the adjoint pair  $(\underline{\rm Ind}, \underline{U})$ coincides with the one associated to the induced $C_2$-action $(T, u)$ on $\underline{\mathcal{A}}$; see \cite[Subsection~4.2]{CCR}. The functor $K$ above coincides with the corresponding comparison functor associated to $(\underline{\rm Ind}, \underline{U})$; see \cite[VI.3]{McLane}. Consequently, we can apply \cite[Proposition~3.5]{Chen15} to infer that $K$ is an equivalence up to retracts. The last statement follows from Lemma~\ref{lem:e-uptore}.
\end{proof}

Let $\mathcal{T}$ be a triangulated category. A $C_2$-action $(T, u)$ on $\mathcal{T}$ is called a \emph{triangle $C_2$-action} if $T$ is a triangle autoequivalence and $u\colon {\rm Id}_\mathcal{T}\rightarrow T^2$ is an isomorphism between triangle functors; compare \cite[Subsection~6.1]{CCR}.

The following result allows us to deal with equivariant triangulated categories.

\begin{lem}\label{lem:tri}
Let $\mathcal{T}$ be a triangulated category with a triangle $C_2$-action $(T, u)$. Assume that $2$ is invertible in $\mathcal{T}$. Then the category $\mathcal{T}^{(T, u)}$ is naturally a pre-triangulated category, that is, a triangulated category possibly without the octahedral axiom.
\end{lem}

\begin{proof}
    The result is essentially due to \cite[Theorem~4.1 and Corollary~4.3]{Bal}; compare \cite[Lemma~4.4]{Chen15} and \cite[Theorem~6.9]{Ela}. Here, we mention that the idempotent-split assumption in \cite{Bal} is superfluous, since we might replace $\mathcal{T}$ with its idempotent completion \cite{BS}, which  carries an extended  $C_2$-action. Moreover, a triangle in $\mathcal{T}^{(T, u)}$ is exact if and only if the underlying triangle in $\mathcal{T}$ is exact.
\end{proof}

We mention that, in most cases, the category $\mathcal{T}^{(T, u)}$ is triangulated; see \cite[Corollary~6.10]{Ela}. The following remark will be used implicitly  in the sequel.

\begin{rem}
    Assume that $\mathcal{A}$ is a Frobenius exact category with an exact $C_2$-action $(T, u)$. The stable category $\underline{\mathcal{A}}$ is canonically triangulated \cite[I.2]{Hap}. Moreover, the induced $C_2$-action $(T, u)$ on $\underline{\mathcal{A}}$ is a triangle $C_2$-action. Assume further that  $2$ is invertible in $\mathcal{A}$. By Lemma~\ref{lem:tri}, the category $(\underline{\mathcal{A}})^{(T, u)}$ is pretriangulated.

    On the other hand, the category  $\mathcal{A}^{(T, u)}$ of $C_2$-equivariant objects in $\mathcal{A}$ is naturally a Frobenius exact category. Therefore, its stable category $\underline{\mathcal{A}^{(T, u)}}$ is triangulated. In this situation, the comparison functor
    $$K\colon \underline{\mathcal{A}^{(T, u)}}  \longrightarrow  (\underline{\mathcal{A}})^{(T, u)}$$
    is a triangle  funtor, and thus a triangle equivalence up to retracts by Proposition~\ref{prop:comp}.
\end{rem}

\section{Super modules and equivariantization}

In  this section, we study super modules and the relevant categories of $C_2$-equivariant objects.

\subsection{} Let $R=R_{\bar{0}}\oplus R_{\bar{1}}$ be a \emph{super ring}, that is, a $\mathbb{Z}_2$-graded ring. A \emph{super module} over $R$ is a $\mathbb{Z}_2$-graded module. It is often denoted by $M=M_{\bar{0}}\oplus M_{\bar{1}}$, where $M_{\bar{0}}$ is the even part and $M_{\bar{1}}$ is the odd part. Denote by $R\mbox{-SMod}$ the abelian category of all super $R$-modules, whose morphisms are degree-preserving module homomorphisms. Here and later, we use the capital letter ``S" to stand for ``super". For graded rings, we refer to \cite{NVO}.

Let $M$ be a super $R$-module. We denote by $M(\bar{1})$ the \emph{shifted super module}, whose grading is given such that $M(\bar{1})_{\bar{0}}=M_{\bar{1}}$ and $M(\bar{1})_{\bar{1}}=M_{\bar{0}}$. This gives rise to the \emph{degree-shifting automorphism}
$$(\bar{1})\colon R\mbox{-SMod}\longrightarrow R\mbox{-SMod}.$$
Since the square of $(\bar{1})$ equals the identity functor, we obtain a strict $C_2$-action $((\bar{1}), {\rm Id})$ on $R\mbox{-SMod}$.

 Let $(M, \alpha)$ be a $C_2$-equivariant object in $R\mbox{-}{\rm SMod}^{((\bar{1}), {\rm Id})}$, where  $M=M_{\bar{0}}\oplus M_{\bar{1}}$ is  a  super $R$-module and $\alpha\colon M\rightarrow M(\bar{1})$ is an isomorphism satisfying $\alpha(\bar{1})\circ \alpha={\rm Id}_M$.  We define $\Phi(M, \alpha)=M_{\bar{0}}$, which carries a left  $R$-action $*$ as follows. For $a\in R_{\bar{0}}$ and $x\in M_{\bar{0}}$, the action $a_*x$ is given by the original action $ax$ on $M$; for $a\in R_{\bar{1}}$ and $x\in M_{\bar{0}}$, the action $a_*x$ is defined to be $\alpha(ax)$, where $ax$ belongs to $M_{\bar{1}}$. This gives rise to a well-defined functor
 $$\Phi\colon  R\mbox{-}{\rm SMod}^{((\bar{1}), {\rm Id})}\longrightarrow R\mbox{-Mod}, \; (M, \alpha)\longmapsto  M_{\bar{0}}.$$

The following result is contained in \cite[Proposition~5.2]{CCZ}; compare \cite[Theorem~6.4.1]{NVO}.

\begin{prop}\label{prop:SM-M}
Let $R$ be a super ring. Then the functor $\Phi$ above is an equivalence. Moreover, if $2$ is invertible in $R$, it induces an equivalence
 $$\Phi\colon  R\mbox{-}\underline{\rm SMod}^{((\bar{1}), {\rm Id})} \stackrel{\sim}\longrightarrow R\mbox{-}\underline{\rm Mod}.$$
\end{prop}

\begin{proof}
    We describe a quasi-inverse of $\Phi$. Let $X$ be an ordinary $R$-module. Consider a $\mathbb{Z}_2$-graded abelian group $\mathcal{E}(X)=Xe_{\bar{0}}\oplus Xe_{\bar{1}}$, where $e_{\bar{0}}$ is even and $e_{\bar{1}}$ is odd. It becomes a super $R$-module by the following action: for $a\in R_{\bar{i}}$ and $xe_{\bar{j}}$, we have
$$a (xe_{\bar{j}})= (ax)e_{\bar{i}+\bar{j}}.$$
Moreover, we have a canonical isomorphism
$${\rm can}_X \colon \mathcal{E}(X)\longrightarrow \mathcal{E}(X)(\bar{1})$$
of super $R$-modules, which sends $xe_{\bar{i}}$ to $xe_{\bar{i}+\bar{1}}$. In summary, we have the required $C_2$-equivariant object $\Phi^{-1}(X)=(\mathcal{E}(X), {\rm can}_X)$ in $R\mbox{-}{\rm SMod}^{((\bar{1}), {\rm Id})}$.

  Recall from Lemma~\ref{lem:idem-R} that the stable category  $R\mbox{-\underline{\rm Mod}}$ is idempotent-split.   By the equivalence $\Phi$, so is $\underline{R\mbox{-}{\rm SMod}^{((\bar{1}), {\rm Id})}}$. Since $2$ is invertible in $R$, it is invertible in $R\mbox{-Mod}$. Then the induced equivalence follows from Proposition~\ref{prop:comp}.
\end{proof}

Let $\sigma$ be an automorphism of a ring $A$. For each $A$-module $X$, we have the \emph{twisted $A$-module} $^\sigma(X)$. Its typical element is denoted by $^\sigma(x)$, whose $A$-action is given by
$$a\; {^\sigma(x)}={^\sigma(\sigma(a)x)}.$$
This gives rise to the \emph{twisting autoequivalence} $^\sigma(-)$ on $A\mbox{-Mod}$.

Let $R=R_{\bar{0}}\oplus R_{\bar{1}}$ be a super ring. As an ordinary ring $R$, it carries the \emph{parity automorphism} $$g\colon R\longrightarrow R, \; g(a)=(-1)^i a \mbox{ for } a\in R_{\bar{i}}.$$
This automorphism induces a $C_2$-action $({^g(-)}, u)$ on $R\mbox{-Mod}$. Here, for any ordinary $R$-module $X$, we have the canonical isomorphism
\begin{align}\label{equ:u}
u_X\colon X \stackrel{\sim}\longrightarrow {^g{(^g(X))}}, \; x\longmapsto {^g(^g(x))}.
\end{align}

Let $M=M_{\bar{0}}\oplus M_{\bar{1}}$ be a super $R$-module. We have the \emph{parity isomorphism} on $M$
\begin{align}\label{equ:parity}
\alpha_M \colon M\longrightarrow {^g(M)}, \; m\longmapsto (-1)^i\cdot {^g(m)} \mbox{ for } m\in M_{\bar{i}},
\end{align}
which is an isomorphism between ordinary $R$-modules. Furthermore, the pair $(M, \alpha_M)$ becomes an object in $R\mbox{-{\rm Mod}}^{({^g(-)}, u)}$. The assignment defines a functor
$$\Psi \colon R\mbox{-{\rm SMod}} \longrightarrow R\mbox{-{\rm Mod}}^{({^g(-)}, u)}, \; M\longmapsto (M, \alpha_M).$$

The following result is contained in \cite[Proposition~5.7]{CCZ}.

\begin{prop}\label{prop:M-SM}
    Let $R$ be a super ring. Assume that $2$ is invertible in $R$.  Then the functor $\Psi$ above is an equivalence. Moreover, it induces an equivalence
    $$\Psi\colon  R\mbox{-\underline{\rm SMod}} \stackrel{\sim}\longrightarrow R\mbox{-\underline{\rm Mod}}^{({^g(-)}, u)}. $$
\end{prop}

\begin{proof}
We describe a quasi-inverse of $\Psi$. Let $(X, \alpha)$ be an object in $R\mbox{-{\rm Mod}}^{({^g(-)}, u)}$. Set $X_{\bar{0}}=\{x\in X\; |\; \alpha(x)={^g(x)}\}$ and $X_{\bar{1}}=\{x\in X\; |\; \alpha(x)=-{^g(x)}\}$. For each element $y\in X$ with $\alpha(y)={^g(y')}$, we observe that $\frac{1}{2}(y+y')$ belongs to $X_{\bar{0}}$  and that  $\frac{1}{2}(y-y')$ belongs to $X_{\bar{1}}$. We infer that $X=X_{\bar{0}}\oplus X_{\bar{1}}$; moreover, this makes $X$ a super $R$-module. We just set $\Psi^{-1}(X, \alpha)$ to be this super $R$-module. For the induced equivalence, we just apply the same argument as the one in the second paragraph in the proof of Proposition~\ref{prop:SM-M}.
\end{proof}

We mention that, in a certain sense,  the equivalences in Propositions~\ref{prop:SM-M} and \ref{prop:M-SM} are adjoint to each other; see \cite[Remark~5.8]{CCZ} and compare \cite[Theorem~4.4]{DGNO}.

\subsection{} Let $A$ be a ring. A complex $P^\bullet$ of projective $A$-modules is totally acyclic if it is acyclic and the Hom-complex ${\rm Hom}_A(P^\bullet, Q)$ is acyclic for any projective $A$-module $Q$. An $A$-module $G$ is called \emph{Gorenstein projective} \cite{EJ} if there exists a totally acyclic complex $P^\bullet$ whose zeroth cocycle $Z^0(P^\bullet)$ is isomorphic to $G$. Denote by $A\mbox{-GProj}$ the full subcategory of $A\mbox{-Mod}$ formed by Gorenstein projective $A$-modules.

We observe that projective modules are Gorenstein projective. The full subcategory $A\mbox{-GProj}$ is closed under extensions, and thus becomes an exact category in the sense of Quillen \cite{Qui}. Moreover, by \cite[Proposition~3.8]{Bel} it is a Frobenius exact category, whose projective-injective objects are precisely projective $A$-modules. Consequently, by a general result in \cite[I.2]{Hap} its stable category $A\mbox{-\underline{\rm GProj}}$ is canonically triangulated. It is well known that $A\mbox{-\underline{\rm GProj}}$ is closed under direct summands in $A\mbox{-\underline{\rm Mod}}$. It follows from Lemma~\ref{lem:idem-R} that $A\mbox{-\underline{\rm GProj}}$ is idempotent-split.

A totally acyclic complex $P^\bullet$ is \emph{locally-finite} if each component $P^i$ is finitely generated. An $A$-module $G$ is called \emph{totally-reflexive}, if  $G$ is isomorphic to $Z^0(P^\bullet)$ for some locally-finite totally acyclic complex $P^\bullet$; compare \cite{ABr} and \cite[Chapter~8]{LW}. Denote by $A\mbox{-Gproj}$ the full subcategory formed by totally-reflexive $A$-modules. The following fact justifies the notation: if $A$ is left coherent, a Gorenstein-projective $A$-module is totally-reflexive if and only if it is finitely presented; see \cite[Lemma~3.4]{Chen11}.

The category $A\mbox{-Gproj}$  is also a Frobenius exact category, whose projective-injective objects are precisely finitely generated projective $A$-modules. The stable category $A\mbox{-\underline{Gproj}}$ is canonically triangulated, which might viewed as a triangulated subcategory of $A\mbox{-\underline{\rm GProj}}$. In contrast to $A\mbox{-\underline{GProj}}$, the stable category $A\mbox{-\underline{Gproj}}$ is not idempotent-split in general; see \cite[Section~2.5]{Br-Thesis}; compare \cite[the ninth paragraph in p.207]{Orl11}.

The above consideration applies well to super modules. Let $R=R_{\bar{0}}\oplus R_{\bar{1}}$ be a super ring. Denote by $R\mbox{-SGProj}$ the category of \emph{super Gorensten-projective $R$-modules}, and by $R\mbox{-\underline{SGProj}}$ its stable category. Similarly, the category of \emph{super totally-reflexive $R$-modules} is denoted by $R\mbox{-SGproj}$, whose stable category is denoted by $R\mbox{-\underline{SGproj}}$.

\begin{lem}\label{lem:SG}
    Let $M$ be a super $R$-module. Then $M$ is super Gorenstein-projective if and only if it is Gorenstein projective as an ordinary $R$-module.
\end{lem}

\begin{proof}
    Consider the forgetful functor $U\colon R\mbox{-SMod}\rightarrow R\mbox{-Mod}$. The assignment $X\mapsto \mathcal{E}(X)=Xe_{\bar{0}}\oplus Xe_{\bar{1}}$ in the proof of Proposition~\ref{prop:SM-M} yields a functor $\mathcal{E}\colon R\mbox{-Mod}\rightarrow R\mbox{-SMod}$. It is well known that both $(\mathcal{E}, U)$ and $(U, \mathcal{E})$ are adjoint pairs. Both functors $U$ and $\mathcal{E}$ are faithful. Now, the statement follows from a general result \cite[Theorem~3.2]{Chen-Ren} about faithful Frobenius functors.
\end{proof}

\begin{prop}\label{prop:SGProj-GProj}
    Let $R$ be a super ring. Assume that $2$ is invertible in $R$. Then we have two triangle equivalences
 $$\Phi \colon  R\mbox{-}\underline{\rm SGProj}^{((\bar{1}), {\rm Id})} \stackrel{\sim}\longrightarrow R\mbox{-}\underline{\rm GProj}, \mbox{ and } \Psi \colon R\mbox{-\underline{\rm SGProj}} \stackrel{\sim}\longrightarrow R\mbox{-\underline{\rm GProj}}^{({^g(-)}, u)}.$$
\end{prop}

\begin{proof}
    We only prove the equivalence on the left side. Recall  the  equivalence $\Phi$ in Proposition~\ref{prop:SM-M}. We make the following observation: for an object $(M, \alpha)$ in  $R\mbox{-}{\rm SMod}^{((\bar{1}), {\rm Id})}$, the super $R$-module $M$ is super Gorenstein-projective if and only if the $R$-module $\Phi(M, \alpha)=M_{\bar{0}}$ is Gorenstein projective. Then we obtain a restricted equivalence
    \begin{align}\label{equ:Phi-res}
    \Phi\colon R\mbox{-{\rm SGProj}}^{((\bar{1}), {\rm Id})} \stackrel{\sim}\longrightarrow R\mbox{-{\rm GProj}}.
    \end{align}

    We have an isomorphism of ordinary $R$-modules
    $$\phi \colon M_{\bar{0}}\oplus {^g(M_{\bar{0}})}\stackrel{\sim}\longrightarrow M,$$
    which satisfies $\phi(x)=x+\alpha(x)$  and $\phi({^g(x)})=x-\alpha(x)$ for $x\in M_{\bar{0}}$.
    Using Lemma~\ref{lem:SG} and the fact that $R\mbox{-GProj}$ is closed under direct summands, we infer the required observation.

Recall that $R\mbox{-\underline{\rm GProj}}$ is idempotent-split. By the restricted equivalence $\Phi$ above,  so is the stable category $\underline{R\mbox{-SGProj}^{((\bar{1}), {\rm Id})}}$. We apply Proposition~\ref{prop:comp} to
    identify  $\underline{  R\mbox{-{\rm SGProj}}^{((\bar{1}), {\rm Id})}}$ with $R\mbox{-}\underline{\rm SGProj}^{((\bar{1}), {\rm Id})}$. Then we infer the required equivalence $\Phi$.

    It remains to show that $\Phi$ is a triangle functor. By \cite[I.2.8]{Hap}, this follows from the fact that $\Phi$ is induced from an exact functor $
    R\mbox{-{\rm SGProj}}^{((\bar{1}), {\rm Id})} \rightarrow R\mbox{-{\rm GProj}}$.
\end{proof}

In the following remark, we deal with analogous results of Proposition~\ref{prop:SGProj-GProj} for totally-reflexive modules. Since in general,  the stable category of these modules is not idempotent-split, the situation is more subtle.

\begin{rem}
    Keep the same assumptions above. The equivalence (\ref{equ:Phi-res}) restricts further to an equivalence
\begin{align*}
    \Phi\colon R\mbox{-{\rm SGproj}}^{((\bar{1}), {\rm Id})} \stackrel{\sim}\longrightarrow R\mbox{-{\rm Gproj}}.
\end{align*}
However, passing to the stable categories, we obtain the following diagram of functors.
\begin{align}\label{equ:Phi-Gproj}
R\mbox{-\underline{\rm SGproj}}^{((\bar{1}), {\rm Id})}  \stackrel{K}\longleftarrow   \underline{R\mbox{-{\rm SGproj}}^{((\bar{1}), {\rm Id})}}  \stackrel{\underline{\Phi}}\longrightarrow R\mbox{-\underline{\rm Gproj}}.
\end{align}
Here, the comparison functor $K$ in Proposition~\ref{prop:comp} is an equivalence up to retracts, and the induced functor $\underline{\Phi}$ is still an equivalence. In other words, the triangle equivalence $\Phi$ in Proposition~\ref{prop:SGProj-GProj} has to be replaced by the diagram (\ref{equ:Phi-Gproj}) above.

Similarly, the triangle equivalence $\Psi$ in Proposition~\ref{prop:SGProj-GProj} is replaced by the following composition
\begin{align}\label{equ:Psi-Gproj}
R\mbox{-\underline{\rm SGproj}} \stackrel{\underline{\Psi}} \longrightarrow  \underline{R\mbox{-{\rm Gproj}}^{({^g(-)}, u)}} \stackrel{K} \longrightarrow R\mbox{-\underline{\rm Gproj}}^{({^g(-)}, u)},
\end{align}
which consists of an induced equivalence $\underline{\Psi}$ and the comparison functor $K$.
\end{rem}

\section{Projective-module factorizations}\label{sec:PF}

Throughout this section, we fix an nc-triple $(A, \omega, \sigma)$, which consists of  a ring $A$, an element $\omega\in A$ and an automorphism $\sigma$ of $A$ such that
$$\sigma(\omega)=\omega  \mbox{ and } \omega a=\sigma(a) \omega $$
for all $a\in A$. We will relate projective-module factorizations to super Gorenstein-projective modules; see Proposition~\ref{prop:Theta}.

For any $A$-module $M$, we have a functorial homomorphism
$$\omega_M\colon M\longrightarrow {^\sigma(M)}, \; m\longmapsto {^\sigma(\omega m)}.$$
Recall from \cite{Chen24} that a \emph{module factorization} of $\omega$ over $A$ is a quadruple
$$M^\bullet=(M^0, M^1; d_M^0, d_M^1)$$
which consists of two $A$-modules $M^0$ and $M^1$, two homomorphisms $d_M^0\colon M^0 \rightarrow M^1$ and $d_M^1\colon M^1\rightarrow {^\sigma(M^0)}$ satisfying
$$d_M^1\circ d_M^0=\omega_{M^0} \mbox{ and } {^\sigma(d_M^0)}\circ d_M^1=\omega_{M^1}.$$
A morphism $f^\bullet=(f^0, f^1)\colon M^\bullet \rightarrow N^\bullet$ between two module factorizations is given by $A$-module homomorphims $f^i\colon M^i\rightarrow N^i$, which satisfy
$$ d_N^0\circ f^0=f^1\circ d_M^0 \mbox{ and } d_N^1\circ f^1={^\sigma(f^0)}\circ d_M^1.$$
These data form the category $\mathbf{F}(A; \omega)$ of module factorizations.

We fix another automorphism $\tau$ of $A$ satisfying $\tau^2=\sigma$.  For a module factorization $M^\bullet$, we define a new module factorization
$${\rm ST}(M^\bullet) = ({^{{\tau}^{-1}}(M^1)}, {^\tau(M^0)}; {^{\tau^{-1}}(d_M^1)},  {^\tau(d_M^0)}).$$
Here, we identify $^{\tau^{-1}}(^\sigma(M^0))$ with $^\tau(M^0)$, and $^\tau(M^1)$ with $^\sigma({^{\tau^{-1}}(M^1)})$. This gives rise to the \emph{swap-twisting autoequivalence}
\begin{align}\label{equ:ST}
{\rm ST}\colon \mathbf{F}(A; \omega)\longrightarrow \mathbf{F}(A; \omega), \; M^\bullet \longmapsto {\rm ST}(M^\bullet).
\end{align}

Consider the skew polynomial ring $A_1=A[x; \tau]$ in one variable. In particular, we have $xa=\tau(a)x$ for $a\in A$.  The following fact is well known.

\begin{lem}\label{lem:A_1-A}
    Let $V$ be any $A_1$-module. Then we have the following exact sequence of $A_1$-modules.
    $$0\longrightarrow A_1\otimes_A V \stackrel{\partial}  \longrightarrow  A_1\otimes_A {^\tau(V)} \stackrel{\pi} \longrightarrow {^{\tau_1}(V)}\longrightarrow 0$$
    Here, $\partial(1\otimes_A v)=x\otimes_A {^\tau(v)} - 1\otimes_A {^\tau(xv)}$ and $\pi(1\otimes_A {^\tau v})={^{\tau_1}(v)}$. Consequently, the $A_1$-module  $V$ has finite projective dimension if and only if so does the underlying $A$-module $V$.
\end{lem}

\begin{proof}
    The sequence above splits as a sequence of abelian groups. The map $\pi$ has a section $s$, which sends ${^{\tau_1}(v)}$ to $1\otimes_A {^\tau(v)}$. The map $\partial$ admits a retract $r$, which sends $x^i\otimes_A {^\tau(v)}$ to $\sum_{j=0}^{i-1} x^{j}\otimes_A x^{i-1-j}v$ for $i\geq 1$, and sends $1\otimes_A {^\tau(v)}$ to $0$.
\end{proof}

Consider the quotient ring $A_1/(x^2-\omega)$, which is  a \emph{noncommutative double branched cover} of $A$; see \cite{Kno, MU21}. It becomes a super ring such that the element $x$ is odd, and its even part is identified with $A$.

For a module factorization $M^\bullet$, we define a super module $\Theta(M^\bullet)$ over $A_1/(x^2-\omega)$ as follows.  As a super $A$-module, we have
\begin{align}\label{equ:Theta}
\Theta(M^\bullet)=M^0\oplus {^{\tau^{-1}}(M^1)},
\end{align}
where $M^0$ is the even part and $ {^{\tau^{-1}}(M^1)}$ is the odd part. The action of $x$ on it is defined such that
$$x m^0={^{\tau^{-1}}(d_M^0(m^0))} \mbox{ and } x {^{\tau^{-1}}(m^1)}=m'$$
for $m^0\in M^0$, $m^1\in M^1$ and $d_M^1(m^1)={^\sigma(m')}$ with $m'\in M^0$. We mention that $x^2-\omega$ vanishes on $\Theta(M^\bullet)$. Therefore, the super $A_1/(x^2-\omega)$-module is well defined. This gives rise to a functor
$$\Theta\colon \mathbf{F}(A; \omega)\longrightarrow A_1/(x^2-\omega)\mbox{-SMod}.$$

The idea of the following result goes back to \cite[Proposition~2.1]{Kno}.

\begin{lem}
The functor $\Theta$ above is an equivalence.
\end{lem}

\begin{proof}
    We only describe a quasi-inverse of $\Theta$. Let $V=V_{\bar{0}}\oplus V_{\bar{1}}$ be a super module over $A_1/(x^2-\omega)$. In particular, both $V_{\bar{0}}$ and $V_{\bar{1}}$ are naturally $A$-modules. We have a module factorization of $\omega$ as follows:
    $$\Theta^{-1}(V)=(V_{\bar{0}}, {^\tau(V_{\bar{1}})}; d^0, d^1)$$
    where $d^0(v_0)={^\tau(xv_0)}$ and $d^1({^\tau(v_1)})={^\sigma(xv_1)}$ for $v_i\in V_{\bar{i}}$.
\end{proof}

For a module factorization $M^\bullet$, we have a canonical isomorphism
\begin{align}\label{equ:c}
c_{M^\bullet}\colon M^\bullet \longrightarrow {\rm ST}\circ {\rm ST}(M^\bullet)
\end{align}
by identifying $M^0$ with ${^{\tau^{-1}}({^\tau(M^0)})}$, $M^1$ with ${^\tau({^{\tau^{-1}}(M^0)})}$.
This endows $\mathbf{F}(A; \omega)$ with a $C_2$-action $({\rm ST}, c)$.

On the other hand, the category  $A_1/(x^2-\omega)\mbox{-SMod}$ carries a strict $C_2$-action $((\bar{1}), {\rm Id})$ given by the degree-shifting endofunctor. We have a canonical isomorphism
\begin{align}\label{equ:xi}
\xi_{M^\bullet} \colon \Theta \circ {\rm ST}(M^\bullet) \stackrel{\sim}\longrightarrow \Theta(M^\bullet)(\bar{1}),
\end{align}
by identifying ${^{\tau^{-1}}(^\tau(M^0))}$ with $M^0$.

\begin{lem}\label{lem:xi}
  Consider the two $C_2$-actions above. Then the pair $(\Theta; \xi)$ is a $C_2$-equivariant equivalence from $\mathbf{F}(A; \omega)$ to $A_1/(x^2-\omega)\mbox{-}{\rm SMod}$.
\end{lem}

\begin{proof}
    It suffices to verify that the following composition
    $$\Theta(M^\bullet) \xrightarrow{\Theta(c_{M^\bullet})} \Theta \circ {\mathrm{ST}}\circ {\rm ST}(M^\bullet) \xrightarrow{\xi_{{\rm ST}(M^\bullet)}}  (\Theta\circ {\rm ST}(M^\bullet))(\bar{1}) \xrightarrow{\xi_{M^\bullet}(\bar{1})} \Theta(M^\bullet)  $$
    is the identity morphism.
\end{proof}

By a \emph{projective-module factorization} $P^\bullet$, we mean a module factorization whose components $P^0$ and $P^1$ are both projective $A$-modules; see \cite{Dyck, Chen24}. Such factorizations form a full subcategory $\mathbf{PF}(A; \omega)$. It is closed under extensions in $\mathbf{F}(A; \omega)$, and thus becomes an exact category. Moreover, it is a Frobenius exact category, whose projective-injective objects are given by certain \emph{trivial} projective-module factorizations; compare \cite[Proposition~3.11]{Chen24}. The corresponding stable category  is denoted by $\underline{\mathbf{PF}}(A; \omega)$. We mention that the suspension functor $\Sigma$ on $\underline{\mathbf{PF}}(A; \omega)$ sends $P^\bullet$ to
\begin{align}\label{equ:Sigma}
    \Sigma(P^\bullet)=(P^1, {^\sigma(P^0)}; -d^1,-{^\sigma(d^0)}).
\end{align}
Here, the minus sign is consistent with the one appearing in the suspension functor for complexes.

Following \cite{CCKM, MU21}, a \emph{matrix factorization} $P^\bullet$ we mean a projective-module factorization, whose components $P^i$ are both finitely generated. They form a Frobenius exact category ${\mathbf{MF}}(A; \omega)$, whose stable category is denoted by $\underline{\mathbf{MF}}(A; \omega)$.

\begin{prop}\label{prop:Theta}
The equivalence $\Theta$ above restricts to an equivalence
$$\Theta\colon \mathbf{PF}(A; \omega)\stackrel{\sim}\longrightarrow A_1/{(x^2-\omega)}\mbox{-}{\rm SGProj}^{\rm fpd},$$
which further induces a triangle equivalence between the stable categories
$$\Theta\colon \underline{\mathbf{PF}}(A; \omega)\stackrel{\sim}\longrightarrow A_1/{(x^2-\omega)}\mbox{-}\underline{\rm SGProj}^{\rm fpd}.$$
\end{prop}

\begin{proof}
Write $\overline{A_1}=A_1/{(x^2-\omega)}$. It suffices to prove the following claim:  $\Theta(M^\bullet)$ belongs to $\overline{A_1}\mbox{-GProj}^{\rm fpd}$ if and only if both $A$-modules $M^i$ are projective.

By Lemma~\ref{lem:SG}, $\Theta(M^\bullet)$ belongs to $\overline{A_1}\mbox{-SGProj}$ if and only if the ordinary $\overline{A_1}$-module $\Theta(M^\bullet)$ is Gorenstein-projective. The inclusion $A\hookrightarrow \overline{A_1}$ is a split Frobenius extension. By the general result in \cite[Theorem~3.2 and Example~3.8]{Chen-Ren}, the ordinary $\overline{A_1}$-module $\Theta(M^\bullet)$ is Gorenstein-projective if and only if the underlying $A$-module $\Theta(M^\bullet)=M^0\oplus {^{\tau^{-1}}(M^1)}$ is Gorenstein-projective. Recall the well-known fact that any Gorenstein projective $A$-module with finite projective dimension is necessarily projective. Combining the above facts, we infer that  $\Theta(M^\bullet)$ belongs to $\overline{A_1}\mbox{-SGProj}^{\rm fpd}$ if and only if the underlying $A$-module $\Theta(M^\bullet)=M^0\oplus {^{\tau^{-1}}(M^1)}$ is projective. The claim follows immediately.
\end{proof}

\begin{rem}
    The triangle equivalence $\Theta$ above restricts to a triangle equivalence
    $$\underline{\mathbf{MF}}(A; \omega) \stackrel{\sim}\longrightarrow \overline{A_1} \mbox{-}\underline{\rm SGproj}^{\rm fpd}.$$
    Here, $\overline{A_1} \mbox{-}{\rm SGproj}^{\rm fpd}$ denotes the category of totally-reflexive $\overline{A_1}$-modules whose underlying $A$-modules have finite projective dimension.
\end{rem}

\section{The scalar-extension functor} \label{sec:SE}

In this section,  we recall the scalar-extension functor ${\rm SE}$ \cite{Kno, Yos} from Gorenstein-projective modules to projective-module factorizations. The central result is Proposition~\ref{prop:G-PF}, which  claims that ${\rm SE}$ is  $C_2$-equivariant.

We recall the fixed nc-triple $(A, \omega, \sigma)$ and the automorphism $\tau$.  Consider $A_1=A[x; \tau]$ and a new nc-triple $(A_1, x^2-\omega, \sigma_1)$. Here, $\sigma_1$ is the unique automorphism of $A_1$ which extends $\sigma$ and fixes $x$.

Consider a projective-module factorization $P$ of $x^2-\omega$ over $A_1$. Set ${\rm Cok}^0(P^\bullet)$ to be the cokernel of the morphism $d_P^0\colon P^0\rightarrow P^1$. Since $x^2-\omega$ vanishes on  ${\rm Cok}^0(P^\bullet)$, it becomes a module over $A_1/{(x^2-\omega)}$. This gives rise to the (zeroth) \emph{ cokernel functor}
$${\rm Cok}^0\colon \mathbf{PF}(A_1; x^2-\omega)\longrightarrow A_1/{(x^2-\omega)}\mbox{-Mod}.$$

\begin{thm}\label{thm:Eis}
    The above cokernel functor induces a triangle equivalence
    $${\rm Cok}^0\colon \underline{\mathbf{PF}}(A_1; x^2-\omega) \stackrel{\sim}\longrightarrow A_1/{(x^2-\omega)}\mbox{-}\underline{\rm GProj}^{\rm fpd}.$$
\end{thm}

\begin{proof}
    We observe that the element $x^2-\omega$ is \emph{regular} in $A_1$, that is, a non-zero-divisor on each side in $A_1$. Then the required equivalence follows from a general result \cite[Theorem~5.6]{Chen24}; consult the proof of \cite[Theorem~7.2]{Chen24}.
\end{proof}

Assume that $G$ belongs to $A_1/{(x^2-\omega)}\mbox{-GProj}^{\rm fpd}$, that is, $G$ is a Gorenstein-projective module over $A_1/{(x^2-\omega)}$ whose underlying $A$-module has finite projective dimension. Then the underlying $A$-module $G$ is indeed projective; see the proof of Proposition~\ref{prop:Theta}. Consider the following sequence of $A_1$-modules.
\begin{align}\label{seq:partial}
    A_1\otimes_A G  \stackrel{\partial^0} \longrightarrow A_1\otimes_A {^\tau(G)} \stackrel{\partial^1}\longrightarrow {^{\sigma_1}(A_1\otimes_A G)}
\end{align}
Here, we have
$$\partial^0(1\otimes_A b)=x\otimes_A {^\tau(b)}-1\otimes_A {^\tau(xb)}$$
and
$$\partial^1(1\otimes_A {^\tau(b)})={^{\sigma_1}(x\otimes_A b)}+{^{\sigma_1}(1\otimes_A xb)}.$$
The quadruple
$${\rm SE}(G)=(A_1\otimes_A G, A_1\otimes_A {^\tau(G)}; \partial^0, \partial^1)$$
is a projective-module factorization of $x^2-\omega$ over  $A_1$. Therefore, we have the \emph{scalar-extension functor}
$${\rm SE} \colon A_1/{(x^2-\omega)}\mbox{-GProj}^{\rm fpd} \longrightarrow \mathbf{PF}(A_1; x^2-\omega), \; G\longmapsto {\rm SE}(G). $$

The following result is essentially due to \cite[Lemma~2.6~ii)]{Kno}; see also \cite[(12.2)~Lemma]{Yos}.

\begin{thm}\label{thm:Kno}
    The functor ${\rm SE}$ above induces a triangle equivalence
    $${\rm SE}\colon  A_1/{(x^2-\omega)}\mbox{-}\underline{\rm GProj}^{\rm fpd} \stackrel{\sim}\longrightarrow \underline{\mathbf{PF}}(A_1; x^2-\omega).$$
\end{thm}

\begin{proof}
    We observe that ${\rm SE}$ sends short exact sequences of Gorenstein projective $A_1/{(x^2-\omega)}$-modules to short exact sequences of projective-module factorizations. Moreover, it sends projective modules to projective-injective objects in $\mathbf{PF}(A_1; x^2-\omega)$. Consequently, we have the induced functor between the stable categories.

    Let $G$ be in $A_1/{(x^2-\omega)}\mbox{-}{\rm GProj}^{\rm fpd}$. Then ${\rm Cok}^0\circ {\rm SE}(G)$ is defined to be the cokernel of the morphism $\partial^0$ above. In view of Lemma~\ref{lem:A_1-A}, this cokernel is isomorphic to ${^{\tau_1}(G)}$. In other words, the composition functor ${\rm Cok}^0\circ {\rm SE}$ is isomorphic to the twisting autoequivalence ${^{\tau_1}(-)}$ on $A_1/{(x^2-\omega)}\mbox{-}\underline{\rm GProj}^{\rm fpd}$. By Theorem~\ref{thm:Eis} ${\rm Cok}^0$ is an equivalence. Then the required statement follows immediately.
\end{proof}

Recall the parity automorphism $g$ on  $A_1/{(x^2-\omega)}$ given by $g(x)=-x$ and $g(a)=a$ for $a\in A$. Then the category $A_1/{(x^2-\omega)}\mbox{-}{\rm GProj}^{\rm fpd}$ is endowed with a $C_2$-action $(^g(-), u)$; see (\ref{equ:u}).

Denote by ${\rm ST}_1$ the swap-twisting autoequivalence on $\underline{\mathbf{PF}}(A_1; x^2-\omega)$. it gives rise to  a $C_2$-action $({\rm ST}_1, c)$ on $\underline{\mathbf{PF}}(A_1; x^2-\omega)$; compare (\ref{equ:c}).

Take any $G$ from $A_1/{(x^2-\omega)}\mbox{-}{\rm GProj}^{\rm fpd}$. We claim that there is a canonical isomorphism
$$\eta^\bullet_G=(\eta^0, \eta^1)\colon {\rm SE}(^g(G))\stackrel{\sim}\longrightarrow {\rm ST}_1\circ {\rm SE}(G).$$
It is illustrated by the following commutative diagram of $A_1$-modules.
\[\xymatrix{
A_1\otimes_A {^g(G)} \ar[rr]^-{\partial^0} \ar[d]_-{\eta^0} && A_1 \otimes_A {^\tau(^g(G))} \ar[rr]^-{\partial^1} \ar[d]^{\eta^1} && {^{\sigma_1}(A_1\otimes_A {^g(G)})} \ar[d]^-{^{\sigma_1}(
\eta^0)}\\
{^{\tau_1^{-1}}(A_1\otimes_A {^\tau(G)})} \ar[rr]^-{^{\tau^{-1}}(\partial^1)} && {^{\tau_1}(A_1\otimes_A G)} \ar[rr]^-{^\tau(\partial^0)} &&  ^{\sigma_1}(^{\tau_1^{-1}}(A_1\otimes_A {^\tau(G)}))
}\]
The isomorphism $\eta^0$ sends $1\otimes_A {^g(v)}$ to $^{\tau_1^{-1}}(1\otimes_A {^\tau(v)})$, and the isomorphism $\eta^1$ sends $1\otimes_A {^\tau(^g(v))}$ to $^{\tau_1}(1\otimes_A v)$.  Here, differentials $\partial^i$ in the upper row are defined for the twisted module ${^g(G)}$.

\begin{prop}\label{prop:G-PF}
    Consider the two $C_2$-actions above. Then $({\rm SE}, \eta^\bullet)$ is a $C_2$-equivariant equivalence from $A_1/{(x^2-\omega)}\mbox{-}\underline{\rm GProj}^{\rm fpd}$ to $\underline{\mathbf{PF}}(A_1; x^2-\omega)$. Consequently, we have an induced triangle equivalence
    $$({\rm SE}, \eta^\bullet)^{C_2}\colon (A_1/{(x^2-\omega)}\mbox{-}\underline{\rm GProj}^{\rm fpd})^{(^g(-), u)} \stackrel{\sim}\longrightarrow \underline{\mathbf{PF}}(A_1; x^2-\omega)^{({\rm ST}_1, c)}. $$
\end{prop}

\begin{proof}
We observes that the following diagram strictly commutes.
\[\xymatrix{
{\rm SE} \ar[d]_-{{\rm SE}(u)} \ar[rr]^-{c({\rm SE})} && {\rm ST}_1\circ {\rm ST}_1\circ {\rm SE} \\
{\rm SE}\circ {^g(-)}\circ {^g(-)} \ar[rr]^-{\eta^\bullet ({^g(-)})} && {\rm ST}_1\circ {\rm SE} \circ {^g(-)} \ar[u]_-{{\rm ST}_1(\eta^\bullet)}
}\]
Roughly speaking, the commutativity holds since the action of $x$ does not play a role in this diagram.

The commutativitiy above implies that $({\rm SE}, \eta^\bullet)$ is a $C_2$-equivariant functor, and thus a $C_2$-equivariant equivalence by Theorem~\ref{thm:Kno}. The last statement follows from Lemma~\ref{lem:equiva-equiva}.
\end{proof}

Consider the skew polynomial ring $A_2=A_1[y; \tau_1]$. In particular, in $A_2$ we have $yx=xy$ and $ya=\tau(a)y$ for $a\in A$. Denote by $\sigma_2$ the automorphism of $A_2$ given by $\sigma_2(a)=\sigma(a)$, $\sigma_2(x)=x$ and $\sigma_2(y)=y$. We have
$$(y^2-x^2+\omega)z=\sigma_2(z) (y^2-x^2+\omega)$$
for all $z\in A_2$. We view  the quotient ring $A_2/(y^2-x^2+\omega)$ as a super ring  by means of ${\rm deg}(y)=\bar{1}$ and ${\rm deg}(x)=\bar{0}={\rm deg}(a)$.

Let $N^\bullet=(N^0, N^1; d_N^0, d_N^1)$  be a module factorization of $x^2-\omega$ over $A_1$. Similar to (\ref{equ:Theta}),  we define
$$\Theta_1(N^\bullet)=N^0\oplus {^{\tau_1^{-1}}(N^1)},$$
which is naturally a super module over $A_2/(y^2-x^2+\omega)$. Moreover, in view of (\ref{equ:xi}) we have a natural isomorphism  of super $A_2/(y^2-x^2+\omega)$-modules
$$\xi_{N^\bullet}\colon \Theta_1\circ {\rm ST}_1(N^\bullet) \stackrel{\sim }\longrightarrow \Theta_1(N^\bullet)(\bar{1}). $$

Recall that $({\rm ST}_1, c)$ is a $C_2$-action on $\underline{\mathbf{PF}}(A_1; x^2-\omega)$. The pair $((\bar{1}), {\rm Id})$ is a strict $C_2$-action on $A_2/(y^2-x^2+\omega)\mbox{-}\underline{\rm SGProj}^{\rm fpd}$.

\begin{prop}\label{prop:Theta_1}
  Consider the two $C_2$-actions above.  Then the pair $(\Theta_1, \xi)$ is  a $C_2$-equivariant equivalence from $\underline{\mathbf{PF}}(A_1; x^2-\omega)$ to $A_2/(y^2-x^2+\omega)\mbox{-}\underline{\rm SGProj}^{\rm fpd}$. Consequently, we have an induced triangle equivalence
  $$(\Theta_1, \xi)^{C_2}\colon \underline{\mathbf{PF}}(A_1; x^2-\omega)^{({\rm ST}_1, c)} \stackrel{\sim}\longrightarrow (A_2/(y^2-x^2+\omega)\mbox{-}\underline{\rm SGProj}^{\rm fpd})^{((\bar{1}), {\rm Id})}.$$
\end{prop}

\begin{proof}
    The first statement is obtained by applying Lemma~\ref{lem:xi} and Proposition~\ref{prop:Theta} to the nc-triple $(A_1, x^2-\omega, \sigma_1)$. The last statement follows from Lemma~\ref{lem:equiva-equiva}.
\end{proof}

\begin{rem}
    The results in this section work well for matrix factorizations and totally-reflexive modules. To be more precise, the triangle equivalence in Proposition~\ref{prop:G-PF} restricts to the following triangle equivalence
$$(A_1/{(x^2-\omega)}\mbox{-}\underline{\rm Gproj}^{\rm fpd})^{(^g(-), u)} \stackrel{\sim}\longrightarrow \underline{{\mathbf{MF}}}(A_1; x^2-\omega)^{({\rm ST}_1, c)}. $$    The triangle equivalence in Proposition~\ref{prop:Theta_1} restricts to a triangle equivalence
$$\underline{{\mathbf{MF}}}(A_1; x^2-\omega)^{({\rm ST}_1, c)} \stackrel{\sim}\longrightarrow (A_2/(y^2-x^2+\omega)\mbox{-}\underline{\rm SGproj}^{\rm fpd})^{((\bar{1}), {\rm Id})}.$$
\end{rem}

\section{Kn\"{o}rrer's periodicity theorem}\label{sec:Kn}

In this section, we establish noncommutative Kn\"{o}rrer periodicity for projective-module factorizations, and obtain an explicit quasi-inverse.

We fix an nc-triple $(A, \omega, \sigma)$ and form the stable category $\underline{\mathbf{PF}}(A; \omega)$ of projective-module factorizations of $\omega$ over $A$.

Fix another automorphism $\tau$ of $A$ satisfying $\tau^2=\sigma$ and $\tau(\omega)=\omega$. Consider the skew polynomial ring $A_1=A[x; \tau]$. Both automorphisms $\sigma$ and $\tau$ extends naturally to automorphisms $\sigma_1$ and $\tau_1$ on $A_1$ which fix $x$. In particular, we have the intermediate nc-triple $(A_1, x^2-\omega, \sigma_1)$.

Consider the skew polynomial ring $A_2=A_1[y; \tau_1]$ and its element $y^2-x^2+\omega$. The automorphism $\sigma_1$ on $A_1$ extends an automorphism $\sigma_2$ on $A_2$, which fixes $y$.  In other words, we obtain another nc-triple $(A_2, y^2-x^2+\omega, \sigma_2)$. We form the stable category $\underline{\mathbf{PF}}(A_2; y^2-x^2+\omega)$ of projective-module factorizations of $y^2-x^2+\omega$ over $A_2$.

The following result extends the noncommutative Kn\"{o}rrer periodicity between matrix factorizations in \cite[Theorem~5.11]{CKMW} and \cite[Theorem~3.9]{MU} to projective-module factorizations, using $C_2$-equivariantization.

\begin{thm}\label{thm:main}
Keep the assumptions above. Assume further that $2$ is invertible in $A$. Then there is an explicit triangle equivalence
$${\rm Kn}\colon \underline{\mathbf{PF}}(A; \omega)\stackrel{\sim}\longrightarrow \underline{\mathbf{PF}}(A_2; y^2-x^2+\omega),$$
which restricts to an equivalence up to retracts
$$\underline{{\mathbf{MF}}}(A; \omega) \longrightarrow \underline{{\mathbf{MF}}}(A_2; y^2-x^2+\omega).$$
\end{thm}

We refer to (\ref{equ:Kn}) below for the explicit construction of ${\rm Kn}$, and Proposition~\ref{prop:inverse} for an explicit quasi-inverse.

\begin{proof}
    We will divide the proof into six steps, each of which is conceptually easy and contains an explicit triangle equivalence.  The following diagram illustrates the construction of ${\rm Kn}$.

\[\xymatrix@C=13.5pt{
 \underline{\mathbf{PF}}(A; \omega)  \ar[d]^-{\Theta}_-{\rm Prop.~\ref{prop:Theta}} \ar@{.>}[rrrr]^-{\rm Kn} &&&&  \underline{\mathbf{PF}}(A_2; y^2 - x^2 + \omega)\\
  \overline{A_1}\text{-}\underline{\rm SGProj}^{\mathrm{fpd}} \ar[d]^-{\Psi}_-{\rm Prop.~\ref{prop:SGProj-GProj}} &&&&  \overline{A_2}\text{-}\underline{\rm GProj}^{\mathrm{fpd}}  \ar[u]^-{{\rm SE}_1}_-{\rm Thm.~\ref{thm:Kno}}\\
(\overline{A_1}\text{-}\underline{\rm GProj}^{\mathrm{fpd}})^{({^g(-)}, u)}
\ar[rr]^-{(\mathrm{SE}, \eta^\bullet)^{C_2}}_-{\rm Prop.~\ref{prop:G-PF}}  & &  \underline{\mathbf{PF}}(A_1; x^2-\omega)^{(\mathrm{ST}_1, c)} \ar[rr]^-{(\Theta_1, \xi)^{C_2}}_-{\rm Prop.~\ref{prop:Theta_1}} &&
(\overline{A_2}\text{-}\underline{\rm SGProj}^{\mathrm{fpd}})^{((\bar{1}), \mathrm{Id})} \ar[u]^-{\Phi}_-{\rm Prop.~\ref{prop:SGProj-GProj}}
}\]

    Here, we set $\overline{A_1}=A_1/{(x^2-\omega)}$ and $\overline{A_2}=A_2/{(y^2-x^2+\omega)}$. Assume that $P^\bullet=(P^0, P^1; d^0, d^1)$ is an arbitrary object in $\underline{\mathbf{PF}}(A; \omega)$.

\vskip 3pt

    \emph{Step 1.} \; We view $\overline{A_1}$ as a super ring by taking ${\rm deg}(x)=\bar{1}$. Recall that $\overline{A_1}\mbox{-\underline{\rm SGProj}}^{\rm fpd}$ denotes the stable category of super Gorenstein-projective $\overline{A_1}$-modules whose underlying $A$-modules have finite projective dimension.  By Proposition~\ref{prop:Theta}, we have a triangle equivalence
    $$\Theta\colon \underline{\mathbf{PF}}(A; \omega)\stackrel{\sim}\longrightarrow \overline{A_1}\mbox{-\underline{\rm SGProj}}^{\rm fpd}. $$
    Recall that
    $$\Theta(P^\bullet)=P^0\oplus {^{\tau^{-1}}(P^1)}$$
    as a super $A$-module. The action of the element $x$ on $\Theta(P^\bullet)$ is induced by the differentials $d^0$ and $d^1$.
\vskip 3pt

    \emph{Step 2.}\; Denote by $g$ the parity automorphism on $\overline{A_1}$, which is given by $g(x)=-x$ and $g(a)=a$ for $a\in A$. Recall that $(^g(-), u)$ is a $C_2$-action on $\overline{A_1}\mbox{-}\underline{\rm GProj}^{\rm fpd}$. By applying Proposition~\ref{prop:SGProj-GProj} to the super ring $\overline{A_1}$, we obtain a triangle equivalence
    $$\Psi\colon  \overline{A_1}\mbox{-\underline{\rm SGProj}}^{\rm fpd}\stackrel{\sim}\longrightarrow (\overline{A_1}\mbox{-}\underline{\rm GProj}^{\rm fpd})^{(^g(-), u)}.$$
    We have
    $$\Psi\circ \Theta(P^\bullet)=(P^0\oplus {^{\tau^{-1}}(P^1)}, \alpha),$$
    where $P^0\oplus {^{\tau^{-1}}(P^1)}$ is viewed as an ordinary module over $\overline{A_1}$ and $\alpha$ is the parity isomorphism (\ref{equ:parity}) on it. Roughly speaking, $\alpha$ acts by $1$ on $P^0$, and by $-1$ on $^{\tau^{-1}}(P^1)$.
\vskip 3pt

    \emph{Step 3.} \; Recall that the swap-twisting autoequivalence ${\rm ST}_1$ on $\underline{\mathbf{PF}}(A_1; x^2-\omega)$ yields a $C_2$-action $({\rm ST}_1, c)$. By Proposition~\ref{prop:G-PF}, we have a triangle equivalence
    $$({\rm SE}, \eta^\bullet)^{C_2}\colon (\overline{A_1}\mbox{-}\underline{\rm GProj}^{\rm fpd})^{(^g(-), u)}\stackrel{\sim}\longrightarrow  \underline{\mathbf{PF}}(A_1; x^2-\omega)^{({\rm ST}_1, c)}.$$

    To describe the image of $\Psi\circ \Theta(P^\bullet)$ under this equivalence, we recall the scalar-extension functor ${\rm SE}$ in Section~\ref{sec:SE}. Consider the following two projective $A_1$-modules
    $$Q^0=A_1\otimes_A(P^0\oplus {^{\tau^{-1}}(P^1)}) \mbox{ and } Q_1=A_1\otimes_A(^\tau(P^0)\oplus P^1).$$
    We have a morphism $\partial^0\colon Q^0\rightarrow Q^1$ given by
    $$\partial^0(1\otimes_A b^0)=x\otimes_A {^\tau(b^0)}-1\otimes_A d^0(b^0) $$
    and $$\partial^0(1\otimes_A {^{\tau^{-1}}(b^1)})=-1\otimes_A {^\tau(b'^0)}+ x\otimes_A b^1.$$
    Here, we  have $b^i\in P^i$ and  assume that $d^1(b^1)={^\sigma(b'^0)}$. We have another morphism $\partial^1\colon Q^1\rightarrow {^{\sigma_1}(Q^0)}$ given by
    $$\partial^1(1\otimes_A {^\tau(b^0)})= {^{\sigma_1}(x\otimes_A b^0)} + {^{\sigma_1}(1\otimes_A {^{\tau^{-1}}(d^0(b^0))})} $$
    and
    $$\partial^1(1\otimes_A b^1)= {^{\sigma_1}(1 \otimes_A b'^0)} + {^{\sigma_1}(x\otimes_A {^{\tau^{-1}}(b^1)})}. $$
    Here, we still use the convention that $d^1(b^1)={^\sigma(b'^0)}$. These data define an object
    $$Q^\bullet=(Q^0, Q^1; \partial^0, \partial^1)\in \underline{\mathbf{PF}}(A_1; x^2-\omega),$$
    which equals ${\rm SE}(P^0\oplus {^{\tau^{-1}}(P^1)})$.

    Set
    $$\beta^\bullet=\eta^\bullet_{P^0\oplus {^{\tau^{-1}}(P^1)}} \circ {\rm SE}(\alpha)\colon Q^\bullet \longrightarrow {\rm ST}_1(Q^\bullet).$$ Roughly speaking, $\beta^0$ and $\beta^1$ acts  by $1$ on the direct summands $A_1\otimes_A P^0$ and $A_1\otimes_A{^\tau(P^0)}$; they act by $-1$ on $A_1\otimes_A {^{\tau^{-1}}(P^1)}$ and $A_1\otimes_A P^1$.

    In summary, the equivalence $({\rm SE}, \eta^\bullet)^{C_2}$ sends $\Psi\circ \Theta(P^\bullet)$ to the following $C_2$-equivariant object
    $$(Q^\bullet, \beta^\bullet)\in \underline{\mathbf{PF}}(A_1; x^2-\omega)^{({\rm ST}_1, c)}. $$
\vskip 3pt

\emph{Step 4.}\; We view $\overline{A_2}$ as a super ring such that $y$ is odd and $x$ is even. The even part of $\overline{A_2}$ is identified with the ring $A_1$. We have a strict $C_2$-action $((\bar{1}), {\rm Id})$ on $\overline{A_2}\mbox{-}\underline{\rm SGProj}^{\rm fpd}$ induced by the degree-shifting automorphism. By Proposition~\ref{prop:Theta_1}, we have a triangle equivalence
  $$(\Theta_1, \xi)^{C_2}\colon \underline{\mathbf{PF}}(A_1; x^2-\omega)^{({\rm ST}_1, c)} \stackrel{\sim}\longrightarrow (\overline{A_2}\mbox{-}\underline{\rm SGProj}^{\rm fpd})^{((\bar{1}), {\rm Id})}.$$
    We have
    $$\Theta_1(Q^\bullet)=Q^0\oplus {^{\tau_1^{-1}}(Q^1)}$$
    as a ungraded $A_1$-module. As a super $\overline{A}_2$-module, its even part is $Q^0$ and its odd part is ${^{\tau_1^{-1}}(Q^1)}$, respectively. The action of $y$ on $\Theta_1(Q^\bullet)$ is induced by $\partial^0$ and $\partial^1$.

    We have an isomorphism
    $$\gamma=\xi_{Q^\bullet}\circ \Theta_1(\beta^\bullet)\colon \Theta_1(Q^\bullet) \stackrel{\sim}\longrightarrow \Theta_1(Q^\bullet)(\bar{1})$$
    of super $\overline{A_2}$-modules. This gives to a $C_2$-equivariant object $(Q^0\oplus {^{\tau_1^{-1}}(Q^1)}, \gamma)$ in $(\overline{A_2}\mbox{-}\underline{\rm SGProj}^{\rm fpd})^{((\bar{1}), {\rm Id})}$.

    In summary, the equivalence $(\Theta_1, \xi)^{C_2}$ sends $(Q^\bullet, \beta^\bullet)$ to
    $$ (Q^0\oplus {^{\tau_1^{-1}}(Q^1)}, \gamma)\in (\overline{A_2}\mbox{-}\underline{\rm SGProj}^{\rm fpd})^{((\bar{1}), {\rm Id})}. $$
\vskip 3pt

    \emph{Step 5.}\; Applying Proposition~\ref{prop:SGProj-GProj} to the super ring $\overline{A_2}$, we have a triangle equivalence
    $$\Phi\colon (\overline{A_2}\mbox{-}\underline{\rm SGProj}^{\rm fpd})^{((\bar{1}), {\rm Id})}\stackrel{\sim}\longrightarrow \overline{A_2}\mbox{-}\underline{\rm GProj}^{\rm fpd}.$$
    Recall that
    $$\Phi(Q^0\oplus {^{\tau_1^{-1}}(Q^1)}, \gamma)=Q^0$$
    as a ungraded $A_1$-module. The action of $y$ on $Q^0$ is given by the composition of $\partial^0$ and $\gamma$. To be more precise, for any element $b \in Q^0$, we have
    $$yb=\gamma({^{\tau_1^{-1}}(\partial^0(b))})\in Q^0. $$
    Here, we have to notice  that $\gamma$ acts by $-1$ on the direct summand ${^{\tau_1^{-1}}(A_1\otimes_A P^1)}$ of ${^{\tau_1^{-1}}(Q^1)}$. Therefore, we have
    $$y(1\otimes_A b^0)=x\otimes_A b^0+1\otimes_A {^{\tau^{-1}}(d^0(b^0))}$$
    and
    $$y(1\otimes_A {^{\tau^{-1}}(b^1))}= -1\otimes_A b'^0-x\otimes_A {^{\tau^{-1}}(b^1)}$$
for $b^0 \in P^0$, $b^1\in P^1$ and $d^1(b^1)={^\sigma(b'^0)}$.
\vskip 3pt

    \emph{Step 6.}\;  By applying Theorem~\ref{thm:Kno} to the nc-triple $(A_1, x^2-\omega, \sigma_1)$, we obtain a triangle equivalence
    $${\rm SE}_1\colon \overline{A_2}\mbox{-}\underline{\rm GProj}^{\rm fpd} \stackrel{\sim}\longrightarrow \underline{\mathbf{PF}}(A_2; y^2-x^2+\omega),$$
    which is given by the scalar-extension functor.  Here, we observe that $y^2-(x^2-\omega)=y^2-x^2+\omega$.

    Consider the resulting $\overline{A_2}$-module $Q^0$ in Step 5. We will compute ${\rm SE}_1(Q^0)$.  Recall that
    $${\rm SE}_1(Q^0)=(A_2\otimes_{A_1} Q^0, A_2\otimes_{A_1} {^{\tau_1}(Q^0)}; D^0, D^1),$$
    where the differentials $D^i$ play the role of $\partial^i$ in (\ref{seq:partial}).

    We identify $A_2\otimes_{A_1} Q^0$ with $(A_2\otimes_A P^0) \oplus (A_2\otimes_A {^{\tau^{-1}}(P^1)})$, and $A_2\otimes_{A_1} {^{\tau_1}(Q^0)}$ with $(A_2\otimes_A {^\tau(P^0)})\oplus (A_2\otimes_A P^1)$. Using these identifications, the differential
    $$D^0\colon (A_2\otimes_A P^0) \oplus (A_2\otimes_A {^{\tau^{-1}}(P^1)})\longrightarrow (A_2\otimes_A {^\tau(P^0)})\oplus (A_2\otimes_A P^1)$$
    is given by
    $$D^0(1\otimes_A b^0)=y\otimes_A {^\tau(b^0)}-x\otimes_A {^\tau(b^0)}-1\otimes_A d^0(b^0) $$
    and
    $$D^0(1\otimes_A {^{\tau^{-1}}(b^1)})=y\otimes_A b^1+x\otimes_A b^1+1\otimes_A {^\tau(b'^0)}.$$
    Similarly, the differential
    $$D^1\colon (A_2\otimes_A {^\tau(P^0)})\oplus (A_2\otimes_A P^1)\longrightarrow {^{\sigma_2}(A_2\otimes_A P^0)} \oplus {^{\sigma_2}(A_2\otimes_A {^{\tau^{-1}}(P^1)})}$$
    is given by
    $$D^1(1\otimes_A {^\tau(b^0)})={^{\sigma_2}(y\otimes_A b^0)} + {^{\sigma_2}(x\otimes_A b^0)} + {^{\sigma_2}(1\otimes_A {^{\tau^{-1}}(d^0(b^0))})}$$
    and
    $$D^1(1\otimes_A b^1)={^{\sigma_2}(y\otimes_A {^{\tau^{-1}}(b^1)})} - {^{\sigma_2}(x\otimes_A {^{\tau^{-1}}(b^1)})} - {^{\sigma_2}(1\otimes_A b'^0)}.$$
    Here,  we use the convention that $d^1(b^1)={^\sigma(b'^0)}$ twice.

    In summary,  the projective-module factorization ${\rm SE}_1(Q^0)$ of $y^2-x^2+\omega$ over $A_2$ is explicitly given by
    \begin{align}\label{equ:Kn}
        {\rm Kn}(P^\bullet)=(A_2\otimes_A P^0 \oplus A_2\otimes_A {^{\tau^{-1}}(P^1)}, A_2\otimes_A {^\tau(P^0)}\oplus A_2\otimes_A P^1; D^0, D^1).
    \end{align}
    Combining the six steps above, we  obtain the explicit equivalence ${\rm Kn}$.

    For the restricted functor between matrix factorizations, we repeat the six steps above by replacing PF (\emph{resp}., SGProj, GProj) with MF (\emph{resp}., SGproj, Gproj). However, in Step~2 we have to replace $\Psi$ by the composite functor  (\ref{equ:Psi-Gproj}), and in Step~5 we have to replace $\Phi$ by a diagram of the form (\ref{equ:Phi-Gproj}). Due to the possibly failure of idempotent-splitness of the stable categories of matrix factorizations, the restricted functor is only an equivalence up to retracts.
    \end{proof}

\begin{rem}
(1) In the commutative situation, the explicit equivalence ${\rm Kn}$ coincides with the one in \cite[Section~3]{Kno}; see also \cite[(12.8)~Definition]{Yos} and \cite[8.31~Definition]{LW}. In particular, the proof above indicates that the construction of ${\rm Kn}$ in \cite{Kno} can be obtained via explicit calculations in $C_2$-equivariantization.  We mention the work  \cite{Bird} on infinite Kn\"{o}rrer periodicity using the dg method.

(2) We do not have an explicit example where the restricted functor between matrix factorizations is non-dense; compare \cite[Remark~3.10]{MU} and Corollary~\ref{cor:Orl} below.

(3) In view of Theorem~\ref{thm:main}, we suspect that an analogue of \cite[Theorem~1.2]{Br} for projective-module factorizations might hold.
\end{rem}

There is a surjective ring homomorphism
$${\rm equ}\colon A_2\longrightarrow A_1$$
given by ${\rm equ}(x)=x={\rm equ}(y)$, and ${\rm equ}(a)=a$ for $a\in A$. This makes $A_1$  a right $A_2$-module, and thus an $A$-$A_2$-bimodule. This bimodule induces the following tensor  functor
$$A_1\otimes_{A_2}- \colon \underline{\mathbf{PF}}(A_2; y^2-x^2+\omega)\longrightarrow  \underline{\mathbf{PF}}(A; \omega),$$
which sends $(Q^0, Q^1; d_Q^0, d_Q^1)$ to
$$(A_1\otimes_{A_2} Q^0, A_1\otimes_{A_2} Q^1; A_1\otimes_{A_2} d_Q^0, \kappa\circ (A_1\otimes_{A_2} d_Q^1)).$$
Here, $\kappa\colon A_1\otimes_{A_2} {^{\sigma_2}(Q^0)}\rightarrow {^\sigma(A_1\otimes_{A_2} Q^0)}$ is the obvious isomorphism sending  $f\otimes_{A_2} {^{\sigma_2}(q)}$ to $^\sigma(\sigma_1(f)\otimes_{A_2}q)$. Since $A_1$ is infinitely generated as a left $A$-module, the tensor functor $A_1\otimes_{A_2}-$ does not restrict to matrix factorizations.

\begin{prop}\label{prop:inverse}
    Keep the assumptions in Theorem~\ref{thm:main}. Then the tensor functor $A_1\otimes_{A_2}-$ above is a quasi-inverse of ${\rm Kn}$.
\end{prop}

\begin{proof}
    Let $P^\bullet=(P^0, P^1; d^0, d^1)$ be the arbitary projective-module factorization in the proof of Theorem~\ref{thm:main}. In view of (\ref{equ:Kn}), we infer that $A_1\otimes_{A_2} {\rm Kn}(P^\bullet)$ is isomorphic to the following factorization.
   \begin{align}\label{equ:TKn}
   (A_1\otimes_A P^0 \oplus A_1\otimes_A {^{\tau^{-1}}(P^1)}, A_1\otimes_A {^\tau(P^0)}\oplus A_1\otimes_A P^1; \overline{D^0}, \overline{D^1}).
   \end{align}
    The differentials are described as follows. We have
        $$\overline{D^0}(x^i\otimes_A b^0)=-x^i\otimes_A d^0(b^0) \mbox{ and } \overline{D^0}(x^i\otimes_A {^{\tau^{-1}}(b^1)})=2x^{i+1}\otimes_A b^1+x^i\otimes_A {^\tau(b'^0)}.$$
        Here, we recall the convention that $d^1(b^1)={^\sigma(b'^0)}$. We have
         $$ \overline{D^1} (x^i \otimes_A {^\tau(b^0)}) = {^{\sigma}(2x^{i+1}\otimes_A b^0)} + {^{\sigma}(x^i\otimes_A {^{\tau^{-1}} (d^0(b^0))})}, $$
    and
    $$\overline{D^1}(x^i\otimes_A b^1)=- {^{\sigma}(x^i\otimes_A b'^0)}.$$

    We observe the canonical decomposition $A_1=\bigoplus_{i\geq 0}Ax^i$ as an $A$-$A$-bimodule. Using this decomposition, we infer that (\ref{equ:TKn}) is isomorphic to the coproduct of
    \begin{align}\label{equ:TKn1}
        (P^0, P^1; -d^0, -d^1)
    \end{align}
    and
    \begin{align}\label{equ:TKn2}
  (Ax^i\otimes_A P^0 \oplus Ax^{i-1}\otimes_A {^{\tau^{-1}}(P^1)}, Ax^{i-1}\otimes_A {^\tau(P^0)}\oplus Ax^i\otimes_A P^1; \overline{D^0}, \overline{D^1}).
    \end{align}
    for all $i\geq 1$.

    Since $2$ is invertible in $A$, it is easy to prove that the factorization (\ref{equ:TKn2}) is projective-injective.  Indeed, it is isomorphic to the direct sum
    $$\theta^0({^{\tau^{-i}}(P^1)}) \oplus \theta^1({^{\tau^{2-i}}(P^0)})$$
    of  trivial module factorizations in \cite[Example~3.1]{Chen24}. The factorization (\ref{equ:TKn1}) is certainly isomorphic to $P^\bullet$. This implies that $A_1\otimes_{A_2} {\rm Kn}(P^\bullet)$ is isomorphic to $P^\bullet$ in the stable category. The isomorphism is given by a canonical morphism, and thus functorial. This completes the proof.
\end{proof}

We give an example of the  noncommutative Kn\"{o}rrer periodicity above. For further concrete examples, we refer to \cite[Section~6]{CKMW} and \cite[Section~6]{MU}.

\begin{exm}\label{exm:root}
    {\rm Let $\mathbb{K}$ be a field with characteristic different from $2$, and let $A$ be a finite dimensional hereditary $\mathbb{K}$-algebra. Consider the nc-triple $(A, 0, {\rm Id}_A)$. The corresponding category ${\mathbf{MF}}(A; 0)$ of matrix factorizations coincides with the category of $2$-cyclic complexes of finitely generated projective $A$-modules, and the stable category $\underline{\mathbf{MF}}(A; 0)$ coincides with the corresponding homotopy category. In other words, we have $\underline{\mathbf{MF}}(A; 0)=\mathbf{R}(A)$, the \emph{root category} of $A$ in the sense of \cite[Subsection~5.1]{Hap87} and \cite[Section~7]{PX}.

Take  an automorphism $\tau$ of $A$ satisfying $\tau^2={\rm Id}_A$. We have the skew polynomial rings $A_1=A[x; \tau]$ and $A_2=A_1[y; \tau_1]$. We have the nc-triple $(A_2, y^2-x^2, {\rm Id}_{A_2})$ and form the category $\underline{\mathbf{MF}}(A_2; y^2-x^2)$. Here, $y^2-x^2$ is a central element of $A_2$. By Theorem~\ref{thm:main},  we have a triangle equivalence
$${\rm Kn}\colon \mathbf{R}(A)\stackrel{\sim}\longrightarrow \underline{\mathbf{MF}}(A_2; y^2-x^2).$$
Here, we implicitly use the fact that $\mathbf{R}(A)$ is idempotent-split.}
\end{exm}

\section{Comparing two functors}\label{sec:compare}

In this section, we compare the explicit equivalence in Theorem~\ref{thm:main}  with an explicit tensor functor in \cite{Orl}.

Let $R$ be an arbitrary ring. Denote by $\mathbf{D}^b(R\mbox{-Mod})$ the bounded derived category $R\mbox{-Mod}$, and by $\mathbf{K}^b(R\mbox{-Proj})$ the bounded homotopy category of $R\mbox{-Proj}$. By \cite[I.3.3]{Hap}, we view $\mathbf{K}^b(R\mbox{-Proj})$  as a triangulated subcategory of $\mathbf{D}^b(R\mbox{-Mod})$.

Following \cite{Buc, Orl}, the \emph{big singularity category} of $R$ is defined to be the following Verdier quotient category
$$\mathbf{D}'_{\rm sg}(R)=\mathbf{D}^b(R\mbox{-Mod})/{\mathbf{K}^b(R\mbox{-Proj})}.$$
By sending an $R$-module $M$ to the corresponding stalk complex $M$ concentrated in degree zero, we have the  following  well-defined canonical functor
$$Q_R\colon R\mbox{-\underline{Mod}} \longrightarrow \mathbf{D}'_{\rm sg}(R), \; M \longmapsto M.$$

If $R$ is left noetherian, we have the \emph{singularity category} of $R$
$$\mathbf{D}_{\rm sg}(R)=\mathbf{D}^b(R\mbox{-mod})/{\mathbf{K}^b(R\mbox{-proj})}.$$
By \cite[Proposition~1.13]{Orl}, the canonical functor $\mathbf{D}_{\rm sg}(R)\rightarrow \mathbf{D}'_{\rm sg}(R)$ is fully faithful; see also \cite[Remark~3.6]{Chen11}. We view $\mathbf{D}_{\rm sg}(R)$ as a triangulated subcategory of $\mathbf{D}'_{\rm sg}(R)$. In this situation, the functor $Q_R$ restricts to $R\mbox{-\underline{mod}} \rightarrow \mathbf{D}_{\rm sg}(R)$.

Let us come back to the situation in Section~\ref{sec:Kn}. Let $(A, \omega, \sigma)$ be an nc-triple and $\tau$ an automorphism of $A$ satisfying $\tau^2=\sigma$ and $\tau(\omega)=\omega$. Consider the skew polynomial rings $A_1=A[x; \tau]$ and $A_2=A_1[y; \tau_1]$.

Consider the quotient rings $\bar{A}=A/{(\omega)}$ and $\overline{A_2}=A_2/{(y^2-x^2+\omega)}$. The automorphism $\tau$ induces an autormorphism $\bar{\tau}$ on $\overline{A}$. We form another skew polynomial ring
$$B=\bar{A}[x; \bar{\tau}].$$
We view $\bar{A}$ as a subring of $B$.

We have a surjective ring homomorphism
$$\pi\colon \overline{A_2} \longrightarrow B$$
such that $\pi(x)=x$, $\pi(y)=-x$ and $\pi(a)=\overline{a}$ for $a\in A$. Indeed, this homomorphism induces an isomorphism of rings
\begin{align}\label{iso:B}
    \overline{A_2}/(x+y)\stackrel{\sim} \longrightarrow B.
\end{align}

The discussion above endows $B$ with a natural $\overline{A_2}$-$\overline{A}$-bimodule structure. Moreover, its underlying right $\overline{A}$-module is free. Since the element $x+y$ is regular in $\overline{A_2}$, by (\ref{iso:B}) we infer that the left $\overline{A_2}$-module $B$ has projective dimension one.  Consequently, the following triangle functor is well-defined
$$B\otimes_{\bar{A}}-\colon  \mathbf{D}'_{\rm sg}(\bar{A}) \longrightarrow \mathbf{D}'_{\rm sg}(\overline{A_2}).$$
This functor is analogous to the one in \cite[Section~2]{Orl}.

We have the following comparison theorem, which yields an infinite noncommutative analogue of \cite[Theorem~2.1]{Orl}.

\begin{thm}\label{thm:comparison}
    Keep the assumptions above. Assume further that the element $\omega$ is regular in $A$. Then the following diagram commutes up to a natural isomorphism.
    \begin{align}\label{diag:comp}
    \xymatrix{
\underline{\mathbf{PF}}(A; \omega) \ar[d]_-{Q_{\bar{A}}\circ {\rm Cok}^0} \ar[rr]^-{\rm Kn} && \underline{\mathbf{PF}}(A_2; y^2-x^2+\omega) \ar[d]^-{Q_{\overline{A_2}}\circ {\rm Cok}^0}\\
    \mathbf{D}'_{\rm sg}(\bar{A})  \ar[rr]^-{B\otimes_{\bar{A}}-} &&  \mathbf{D}'_{\rm sg}(\overline{A_2})
    }
    \end{align}
    Consequently, if $A$ has finite left global dimension and $2$ is invertible in $A$, then $B\otimes_{\bar{A}} -\colon \mathbf{D}'_{\rm sg}(\bar{A}) \rightarrow \mathbf{D}'_{\rm sg}(\overline{A_2}) $ is a triangle equivalence.
\end{thm}

\begin{proof}
    Let $P^\bullet=(P^0, P^1; d^0, d^1)$ be any projective-module factorization of $\omega$.  In view of (\ref{equ:Sigma}) and (\ref{equ:Kn}), we have $\Sigma^{-1}\circ {\rm Kn}(P^\bullet)$ is given by
    $$( A_2\otimes_A {^{\tau^{-1}}(P^0)} \oplus A_2\otimes_A {^{\sigma^{-1}}(P^1)}, A_2\otimes_A P^0 \oplus A_2\otimes_A {^{\tau^{-1}}(P^1)}; -{^{{\sigma_2}^{-1}}(D^1)}, -D^0).$$
    Here, we identify ${^{{\sigma_2}^{-1}}( A_2\otimes_A {^\tau(P^0)}\oplus A_2\otimes_A P^1 )}$ with the leftmost term above. For the $D^i$, we refer to Step~6 in the proof of Theorem~\ref{thm:main}.

    The $\bar{A}$-module $M={\rm Cok}^0(P^\bullet)$ fits into a short exact sequence of $A$-modules.
    $$0\longrightarrow P^0\stackrel{d^0} \longrightarrow P^1\stackrel{} \longrightarrow M \longrightarrow 0$$
    We observe that $A_1\otimes_A P^\bullet$ belongs to ${\mathbf{PF}}(A_1; \omega)$. We identify $A_1/{(\omega)}$ with $B$. Applying \cite[Lemma~5.9]{MU21}, a noncommutative version of \cite[Proposition~5.1]{Eis},  to $A_1\otimes_A P^\bullet$, we obtain an exact sequence of $B$-modules.
    $$\delta\colon B\otimes_A {^{\sigma^{-1}}(P^1)} \xrightarrow{B\otimes_A {^{\sigma^{-1}}}(d^1)}  B\otimes_A P^0 \xrightarrow{B\otimes_A d^0} B\otimes_A P^1 \longrightarrow B\otimes_{\bar{A}} M\longrightarrow 0$$
    The sequence $\delta$ is a truncated projective resolution of $B\otimes_{\bar{A}} M$.

For each projective $A$-module $P$, we use the isomorphism (\ref{iso:B}) to \emph{replace} $B\otimes_A P$ by the following two-term complex.
$$\overline{A_2}\otimes_A{^{\tau^{-1}}(P)}\xrightarrow{-(x+y)} \overline{A_2}\otimes_A P$$
Here, the element $-(x+y)$ means the unique $\overline{A_2}$-module homomorphism, sending $1\otimes_A {^{\tau^{-1}}(z)}$ to $-(x+y)\otimes_A z$.  We apply this replacement to the three terms in $\delta$, and obtain the following commutative diagram.
\begin{align}\label{diag:replace}
\xymatrix{
\overline{A_2}\otimes_A {^{(\tau\sigma)^{-1}}(P^1)} \ar[d]_-{-(x+y)} \ar[rr]^-{\overline{A_2}\otimes_A {^{(\tau\sigma)^{-1}}(d^1)}}  && \overline{A_2}\otimes_A {^{\tau^{-1}}(P^0)} \ar[d]^-{-(x+y)}\ar[rr]^-{\overline{A_2}\otimes_A {^{\tau^{-1}}(d^0)}} &&  \overline{A_2}\otimes_A {^{\tau^{-1}}(P^1)} \ar[d]^-{-(x+y)}\\
\overline{A_2}\otimes_A {^{\sigma^{-1}}(P^1)} \ar[d] \ar[rr]^-{\overline{A_2}\otimes_A {^{\sigma^{-1}}(d^1)}}  && \overline{A_2}\otimes_A P^0   \ar[d]  \ar[rr]^-{\overline{A_2}\otimes_A d^0}  && \overline{A_2}\otimes_A P^1 \ar[d]  \\
B\otimes_A {^{\sigma^{-1}}(P^1)} \ar[rr]^-{B\otimes_A {^{\sigma^{-1}}(d^1)}} && B\otimes_A P^0 \ar[rr]^-{B\otimes_A d^0} && B\otimes_A P^1
}\end{align}

In (\ref{diag:replace}), the columns are short exact sequences. However, since $\omega$ is nonzero in $\overline{A_2}$, the upper two rows are not complexes. Instead, we have
$$\omega=x^2-y^2=(x-y)(x+y).$$
Consider the following $\overline{A_2}$-module homomorphism
$$\overline{A_2}\otimes_A {^{\sigma^{-1}}(P^1)} \xrightarrow{x-y} \overline{A_2}\otimes_A {^{\tau^{-1}}(P^1)},$$
which sends $1\otimes_A {^{\sigma^{-1}}(z)}$ to $(x-y)\otimes_A {^{\tau^{-1}}(z)}$. Together with this homomorphism and mutiplying the two homomorphisms on the top by minus one,  the upper part of (\ref{diag:replace}) becomes a \emph{quasi-bicomplex} in the sense of \cite[Section~3]{Chen10}. We form the `total complex' as follows.
\begin{align}\label{equ:total}
\xymatrix{
\overline{A_2}\otimes_A {^{(\tau\sigma)^{-1}}(P^1)}  \ar[rr] && \overline{A_2}\otimes_A {^{\tau^{-1}}(P^0)} \oplus \overline{A_2}\otimes_A {^{\sigma^{-1}}(P^1)} \ar[dll]\\
\overline{A_2}\otimes_A P^0 \oplus \overline{A_2}\otimes_A {^{\tau^{-1}}(P^1)} \ar[rr] && \overline{A_2}\otimes_A P^1
}
\end{align}
The slanted arrow is given by the following matrix of homomorphisms.
\begin{align}\label{equ:mat}
\begin{pmatrix}
-(x+y)& & \overline{A_2}\otimes_A {^{\sigma^{-1}}}(d^1)\\
-\overline{A_2}\otimes_A {^{\tau^{-1}}}(d^0) && x-y
\end{pmatrix}
\end{align}
Since the columns of (\ref{diag:replace}) are exact, we infer by an easy spectral sequence argument that the total complex (\ref{equ:total}) is quasi-isomorphic to the bottom row of (\ref{diag:replace}); compare \cite[the second paragraph in the proof of Proposition~3.4]{Chen10}. Therefore, the cokernel of (\ref{equ:mat}) is isomorphic to the first syzygy $\Omega_{\overline{A_2}}(B\otimes_{\bar{A}} M)$ of $B\otimes_{\bar{A}} M$.

The key observation is that the homomorphism (\ref{equ:mat}) is identified with  $-{^{{\sigma_2}^{-1}}(D^1)}$. We emphasize the minus sign appearing here. Consequently,
${\rm Cok}^0\circ \Sigma^{-1}\circ {\rm Kr}(P^\bullet)$
is isomorphic to $\Omega_{\overline{A_2}}(B\otimes_{\bar{A}} M)$.

Recall from \cite[Lemma~2.2.2]{Buc} the well-known fact that, in $\mathbf{D}'_{\rm sg}(\overline{A_2})$, we have a natural isomorphism
$$\Omega_{\overline{A_2}}(B\otimes_{\bar{A}} M)\simeq \Sigma^{-1} (B\otimes_{\bar{A}} M).$$
In other words, we obtain a functorial isomorphism
$${\rm Cok}^0 \circ \Sigma^{-1} \circ {\rm Kr} (P^\bullet)  \simeq \Sigma^{-1} (B \otimes_{\bar{A}} M) = \Sigma^{-1} (B\otimes_{\bar{A}} {\rm Cok}^0(P^\bullet)).$$
Since ${\rm Cok}^0$ is a triangle functor, it commutes with $\Sigma^{-1}$. After cancelling the two $\Sigma^{-1}$'s above, we infer  the required commutativity.

For the consequence, we assume that the left global dimension of $A$ is $d+1$. By \cite[Theorem~2.10]{Chen24}, the quotient ring $\bar{A}$ is left $d$-Gorenstein. Consequently, by \cite[Theorem~6.9]{Bel00}, the canonical functor $Q_{\bar{A}}$ induces a triangle equivalence
$$Q_{\bar{A}}\colon \bar{A}\mbox{-}\underline{\rm GProj} \stackrel{\sim}\longrightarrow \mathbf{D}'_{\rm sg}(\bar{A}).$$
Combining it with Theorem~\ref{thm:Eis}, we infer that the composite functor $Q_{\bar{A}}\circ {\rm Cok}^0$ is an equivalence. Similarly, the right vertical arrow in (\ref{diag:comp}) is also an equivalence. Then the required equivalence follows from the one in Theorem~\ref{thm:main}.
\end{proof}

In a commutative case, we show that the restriction of ${\rm Kn}$ to matrix factorizations is dense.

\begin{cor}\label{cor:Orl}
 Suppose that $A$ is a commutative noetherian ring with finite global dimension.  Assume that $\omega$ is regular and that $2$ is invertible in $A$. Then the explicit functor
    $${\rm Kn}\colon \underline{{\mathbf{MF}}}(A; \omega) \stackrel{\sim}\longrightarrow \underline{{\mathbf{MF}}}(A_2; y^2-x^2+\omega)$$
    is dense, and thus a triangle equivalence.
\end{cor}

\begin{proof}
    The commutative diagram (\ref{diag:comp}) restricts to the following one.
       \begin{align*}
    \xymatrix{
\underline{{\mathbf{MF}}}(A; \omega) \ar[d]_-{Q_{\bar{A}}\circ {\rm Cok}^0} \ar[rr]^-{\rm Kn} && \underline{{\mathbf{MF}}}(A_2; y^2-x^2+\omega) \ar[d]^-{Q_{\overline{A_2}}\circ {\rm Cok}^0}\\
    \mathbf{D}_{\rm sg}(\bar{A})  \ar[rr]^-{B\otimes_{\bar{A}}-} &&  \mathbf{D}_{\rm sg}(\overline{A_2})
    }
    \end{align*}
    The last paragraph in the proof of Theorem~\ref{thm:comparison} shows that the two vertical arrows are equivalences. Then the required equivalence follows from the affine case of  \cite[Theorem~2.1]{Orl}. Here, we use implicitly the assumption that $2$ is invertible.
\end{proof}

\vskip 5pt

\noindent {\bf Acknowledgements.}\; The authors are grateful to Professor Shiquan Ruan and Dr. Qiang Dong for helpful discussion. The project is supported by National Key R$\&$D Program of China (No. 2024YFA1013801) and  National Natural Science Foundation of China (No.s 12325101 and  12131015).

\bibliography{}

\vskip 10pt

 {\footnotesize \noindent  Xiao-Wu Chen\\
 School of Mathematical Sciences, University of Science and Technology of China, Hefei 230026, Anhui, PR China}\\

{ \footnotesize \noindent  Wenchao Wu\\
 School of Mathematical Sciences, University of Science and Technology of China, Hefei 230026, Anhui, PR China}

\end{document}